\documentclass[11pt]{article}

\usepackage{etex}

\usepackage{mathtools}

\usepackage{enumitem}

\usepackage{amsmath,amsthm}

\usepackage{todonotes}

\makeatletter 
\let\orgdescriptionlabel\descriptionlabel
\renewcommand*{\descriptionlabel}[1]{%
	\let\orglabel\label
	\let\label\@gobble
	\phantomsection
	\edef\@currentlabel{#1}%
 
	\let\label\orglabel
	\orgdescriptionlabel{#1}%
}
											
\def\th@plain{%

	\thm@notefont{} % same as heading font
	\itshape % body font
}
\def\th@definition{%

	\thm@notefont{}% same as heading font
	\normalfont % body font
}

\g@addto@macro\th@remark{\thm@headpunct{}}
\g@addto@macro\th@definition{\thm@headpunct{}}
\g@addto@macro\th@plain{\thm@headpunct{}}
\makeatother % change the ''category code" of @ back to 12

\usepackage{amssymb}
 
\usepackage{graphicx}

\usepackage[final]{showkeys} 

 \usepackage{etoolbox}

 \usepackage{wasysym}
 
\usepackage{mathrsfs}
 
\usepackage{mathpazo}

\usepackage[titletoc]{appendix}
\usepackage[doc]{optional}
 \usepackage{soul}

\usepackage{colortbl,booktabs,sectsty,multirow}
\usepackage{xcolor}

\usepackage{empheq}

\definecolor{myblue}{rgb}{.8, .8, 1}
  \newcommand*\mybluebox[1]{%
    \colorbox{myblue}{\hspace{1em}#1\hspace{1em}}}

 \usepackage[obeyspaces,hyphens,spaces]{url}
 
\usepackage{hyperref}
\hypersetup{
    colorlinks=true,       % false: boxed links; true: colored links
    linkcolor=green,          % color of internal links (change box color with linkbordercolor)
    citecolor=green,        % color of links to bibliography
    filecolor=magenta,      % color of file links
    urlcolor=cyan           % color of external links
}

\usepackage[
  open,
  openlevel=2,
  atend,
  numbered
]{bookmark}

\usepackage[capitalize, nameinlink, noabbrev]{cleveref}
\crefname{equation}{}{}
\crefname{chapter}{Chapter}{Chapter}
\crefname{item}{}{items}
\crefname{figure}{Figure}{Figure}
\crefname{theorem}{Theorem}{Theorem}
\crefname{lemma}{Lemma}{}
\crefname{proposition}{Proposition}{Proposition}
\crefname{corollary}{Corollary}{Corollary}
\crefname{definition}{Definition}{Definition}
\crefname{fact}{Fact}{Fact}
\crefname{example}{Example}{Example}
\crefname{algorithm}{Algorithm}{Algorithm}
\crefname{remark}{Remark}{Remark}
\crefname{note}{Note}{Note}
\crefname{notation}{Notation}{Notation}
\crefname{case}{Case}{Case}
\crefname{exercise}{Exercise}{Exercise}
\crefname{question}{Question}{Question}
\crefname{claim}{Claim}{Claim}

\crefname{enumi}{}{}

\usepackage[top= 2cm, bottom = 2 cm, left = 2.2 cm, right= 2.2 cm]{geometry}

 \usepackage{float}

 \usepackage{pgf}

\allowdisplaybreaks
 
\numberwithin{equation}{section}

\theoremstyle{plain}  
\newtheorem{theorem}{Theorem}[section]

\newtheorem{corollary}[theorem]{Corollary}
\newtheorem{fact}[theorem]{Fact}
\newtheorem{lemma}[theorem]{Lemma}
\newtheorem{proposition}[theorem]{Proposition}

\theoremstyle{definition}  
\newtheorem{definition}[theorem]{Definition}
\newtheorem{example}[theorem]{Example}

\newtheorem{question}[theorem]{Question}
\newtheorem{remark}[theorem]{Remark}

\newcommand{\aff}{\ensuremath{\operatorname{aff}}}
\newcommand{\spn}{\ensuremath{{\operatorname{span}}}}

\newcommand{\CCO}[1]{CC{#1}}

\newcommand{\CRO}[1]{CR{#1}}

\providecommand{\norm}[1]{\lVert#1\rVert}
\providecommand{\Norm}[1]{{\Big\lVert}#1{\Big\rVert}}
\providecommand{\NNorm}[1]{{\bigg\lVert}#1{\bigg\rVert}}
\providecommand{\innp}[1]{\langle#1\rangle}
\providecommand{\Innp}[1]{\Big\langle#1\Big\rangle}
\providecommand{\IInnp}[1]{\bigg\langle#1\bigg\rangle}

\begin{document}
 
\title{ \sffamily  On circumcenters of finite sets in Hilbert spaces}

\author{
         Heinz H.\ Bauschke\thanks{
                 Mathematics, University of British Columbia, Kelowna, B.C.\ V1V~1V7, Canada.
                 E-mail: \href{mailto:heinz.bauschke@ubc.ca}{\texttt{heinz.bauschke@ubc.ca}}.},~
         Hui\ Ouyang\thanks{
                 Mathematics, University of British Columbia, Kelowna, B.C.\ V1V~1V7, Canada.
                 E-mail: \href{mailto:hui.ouyang@ubc.ca}{\texttt{hui.ouyang@ubc.ca}}.},~ 
         and Xianfu\ Wang\thanks{
                 Mathematics, University of British Columbia, Kelowna, B.C.\ V1V~1V7, Canada.
                 E-mail: \href{mailto:shawn.wang@ubc.ca}{\texttt{shawn.wang@ubc.ca}}.}
                 }

\date{July 3, 2018}
\maketitle
 
\begin{abstract}
\noindent
A well-known object in classical Euclidean geometry is the
circumcenter of a triangle, i.e., the point that is equidistant
from all vertices.
The purpose of this paper is to provide a systematic study of the
circumcenter of sets containing finitely many points 
in Hilbert space. This is motivated 
by recent works of Behling, Bello Cruz, and
Santos on accelerated versions of the
Douglas--Rachford method. 
We present basic results and properties of the circumcenter. 
Several examples
are provided to illustrate the tightness of various assumptions. 
\end{abstract}

{\small
\noindent
{\bfseries 2010 Mathematics Subject Classification:} 
{Primary 51M05; Secondary 
41A50,
90C25 

}

\noindent{\bfseries Keywords:}
Best Approximation Problem, 
Circumcenter,
Circumcentered Douglas--Rachford Method
}

\section{Introduction and standing assumption}

Throughout this paper,
	\begin{empheq}[box = \mybluebox]{equation*}
		 \text{$\mathcal{H}$ is a real Hilbert space} 
	\end{empheq}
with inner product $\innp{\cdot,\cdot}$ and induced norm $\|\cdot\|$. 
We denote by  $\mathcal{P}(\mathcal{H})$ the set of all nonempty 
subsets of $\mathcal{H}$ containing \emph{finitely many}
elements. 
Assume that 
\begin{empheq}[box=\mybluebox]{equation*}
S=\{x_{1}, x_{2}, \ldots, x_{m}\} \in \mathcal{P}(\mathcal{H}).
\end{empheq}
\emph{The goal of this paper is to provide a systematic study of
the circumcenter of $S$, i.e., of the (unique if it exists) point in the
affine hull of $S$ that is equidistant all points in $S$.}
The classical case in trigonometry or Euclidean geometry arises
when $m=3$ and $\mathcal{H}=\mathbb{R}^2$. 
Recent applications of the circumcenter focus on the
present much more general case. 
Indeed, our work is motivated by recent works of Behling, Bello
Cruz, and Santos (see \cite{BCS2017} and \cite{BCS2018}) on
accelerating the Douglas--Rachford algorithm by employing the
circumcenter of intermediate iterates to solve certain best approximation
problems. 

The paper is organized as follows.
Various auxiliary results are collected in 
\cref{sec:AuxResults} to ease subsequent proofs.
Based on the circumcenter, we introduce our main actor, the 
\emph{circumcenter operator}, in \cref{sec:DefinCircOper}. 
Explicit formulae for the 
circumcenter are provided in Sections \ref{sec:ClosFormuCircOper}
and \ref{sec:SymmFormuCCS} while
\cref{sec:BasiPropCircOper} records some 
basic properties.
In \cref{sec:LimiCircOperSeqSet}, we turn to the behaviour of the
circumcenter when sequences of sets are considered.
\cref{sec:CircThreePoints} deals with the case when the set
contains three points which yields particularly
pleasing results. 
The importance of the circumcenter in the algorithmic work of
Behling et al.\ is explained in \cref{sec:applications}. 
In the final \cref{sec:CircCrossProd}, we return to more 
classical roots of the circumcenter and discuss formulae
involving cross products when $\mathcal{H}=\mathbb{R}^3$. 

The notation employed is standard and largely follows 
\cite{BC2017}.

\section{Auxiliary results} \label{sec:AuxResults}

In this section, we provide various results that will be useful in the sequel.

\subsection{Affine sets}
Recall that a nonempty subset $S$ of $\mathcal{H}$ 
is an \emph{affine subspace} of $\mathcal{H}$ if
$(\forall \rho\in\mathbb{R})$ $\rho S + (1-\rho)S = S$; 
moreover, the smallest affine subspace containing $S$ is the
\emph{affine hull} of $S$, denoted $\aff S$. 

\begin{fact} {\rm \cite[page 4]{R2015}}
Let $S \subseteq \mathcal{H}$ be an affine subspace and let $a
\in \mathcal{H}$. Then the \emph{translate} of $S$ by $a$,  
which is defined by 
\begin{align*}
S + a = \{ x + a ~|~ x \in S \},
\end{align*}
is another affine subspace.
\end{fact}

\begin{definition}
An affine subspace $S$ is said to be \emph{parallel} to an affine subspace $ M $ if $S = M +a $ for some $ a \in \mathcal{H}$.
\end{definition}

\begin{fact} {\rm \cite[Theorem 1.2]{R2015}} 
\label{fac:AffinePointLinearSpace}
	Every affine subspace $S$ is parallel to a unique linear
	subspace $L$, which is given by 
	\begin{align*}
	(\forall y \in S) \quad L = S - y = S - S . 
	\end{align*}
\end{fact}

\begin{definition} \cite[page 4]{R2015}
 The \emph{dimension} 
 of an affine subspace is defined to be the dimension of the 
 linear subspace parallel to it. 
\end{definition}

\begin{fact} 
{\rm \cite[page 7]{R2015}}
\label{fac:AffSubsExpre}
	Let $x_{1}, \ldots, x_{m} \in \mathcal{H}$. Then
	\begin{align*}
	\aff\{x_{1}, \ldots,  x_{m}\}
	=\Big\{\lambda_{1}x_{1}+\cdots +\lambda_{m}x_{m} ~\Big|~\lambda_{1},\ldots,\lambda_{m} \in \mathbb{R} ~\text{and}~\sum^{m}_{i=1} \lambda_{i}=1 \Big\}.
	\end{align*}
\end{fact}

Some algebraic calculations and \cref{fac:AffSubsExpre} yield the 
next result.

\begin{lemma} \label{lem:AffineHull}
	Let $x_{1}, \ldots,x_{m} \in \mathcal{H}$. Then for every
	$i_{0} \in \{2, \ldots, m\}$, we have
	\begin{align*}
	\aff\{x_{1}, \ldots, x_{m}\} 
	&~=x_{1} + \spn \{x_{2}-x_{1}, \ldots, x_{m}-x_{1}\}\\
	&~=x_{i_{0}}+\spn\{x_{1}-x_{i_{0}},\ldots,x_{i_{0}-1}-x_{i_{0}}, x_{i_{0}+1}-x_{i_{0}},\ldots,x_{m}-x_{i_{0}}\}.
	\end{align*} 
\end{lemma}

\begin{definition} 
{\rm \cite[page~6]{R2015}}
Let $x_{0}, x_{1}, \ldots, x_{m} \in \mathcal{H}$. The $m+1$ vectors $x_{0}, x_{1}, \ldots, x_{m}$ are said to be affinely independent if $\aff \{x_{0}, x_{1}, \ldots, x_{m}\}$ is $m$-dimensional.
\end{definition}

\begin{fact}
{\rm \cite[page 7]{R2015}}
\label{fac:AffinIndeLineInd}
Let $x_{1}, x_{2}, \ldots, x_{m} \in \mathcal{H}$. Then $x_{1}, x_{2}, \ldots,x_{m}$ are affinely independent if and only if
$ x_{2}-x_{1}, \ldots, x_{m}-x_{1}$ are linearly independent.

\end{fact}

\begin{lemma} \label{lem:UniqExpreAffIndp}
Let $x_{1}, \ldots, x_{m}$ be affinely independent vectors in $\mathcal{H}$. Let $p \in \aff \{x_{1}, \ldots, x_{m}\}$. Then there exists a unique vector $\begin{pmatrix} \alpha_{1}& \cdots& \alpha_{m}\end{pmatrix}^{\intercal} \in \mathbb{R}^{m}$ with $\sum^{m}_{i=1} \alpha_{i} =1$ such that
\begin{align*}
p= \alpha_{1} x_{1} + \cdots + \alpha_{m} x_{m}.
\end{align*}
\end{lemma}

The following lemma will be useful later. 

\begin{lemma}\label{lem:AffineIndep:OpenSet}
Let
\begin{align*}
\mathcal{O}=\Big\{ (x_{1}, \ldots, x_{m-1}, x_{m}) \in  \mathcal{H}^{m} 
~\Big|~ x_{1}, \ldots, x_{m-1}, x_{m}~\text{are affinely independent} \Big\}.
\end{align*}
Then $\mathcal{O}$ is open.
\end{lemma}

\begin{proof}
Assume to the contrary that there exist $( x_{1}, \ldots,
x_{m-1}, x_{m})\in \mathcal{O}$ such that for every $k \in
\mathbb{N}\smallsetminus\{0\}$, there exist $( x^{(k)}_{1}, \ldots, x^{(k)}_{m-1}, x^{(k)}_{m}) \in B\Big(( x_{1}, \ldots, x_{m-1}, x_{m}); \frac{1}{k} \Big)$ such that $ x^{(k)}_{1}, \ldots, x^{(k)}_{m-1}, x^{(k)}_{m} $ are affinely dependent. By \cref{fac:AffinIndeLineInd}, for every $k$, 
there exists $b^{(k)}=(\beta^{(k)}_{1},
\beta^{(k)}_{2},\ldots,\beta^{(k)}_{m-1}) \in \mathbb{R}^{m-1}
\smallsetminus \{0\}$ such that
\begin{align}\label{eq:linindep:mitem}
\beta^{(k)}_{1}(x^{(k)}_{2}-x^{(k)}_{1})+\cdots+\beta^{(k)}_{m-1}(x^{(k)}_{m}-x^{(k)}_{1})=0.
\end{align}
Without loss of generality we assume
\begin{align}\label{eq:lem:lineindp:1}
(\forall k \in \mathbb{N}\smallsetminus\{0\} ) \quad \norm{b^{(k)}}^{2}=\sum^{m-1}_{i=1}(\beta^{(k)}_{i})^{2}= 1,
\end{align}
and there exists $\bar{b}=(\beta_{1}, \ldots, \beta_{m-1}) \in \mathbb{R}^{m-1}$ such that 
\begin{align*}
\lim_{k \rightarrow \infty} ( \beta^{(k)}_{1}, \ldots, \beta^{(k)}_{m-1} ) =\lim_{k \rightarrow \infty}b^{(k)} =\bar{b}=(\beta_{1}, \ldots, \beta_{m-1}).
\end{align*}
Let $k$ go to infinity in \cref{eq:lem:lineindp:1}, we get 
\begin{align*}
\norm{\bar{b}}^{2} = \beta^{2}_{1}+ \cdots+\beta^{2}_{m-1}=1,
\end{align*}
which yields that $(\beta_{1}, \ldots, \beta_{m-1})\neq 0$.

Let $k$ go to infinity in \cref{eq:linindep:mitem}, we obtain
\begin{align*}
\beta_{1}(x_{2}-x_{1})+\cdots+\beta_{m-1}(x_{m}-x_{1})=0, 
\end{align*}
which means that $x_{2}-x_{1}, \ldots, x_{m}-x_{1}$ are linearly dependent. By \cref{fac:AffinIndeLineInd},
it contradicts with the assumption that $x_{1}, \ldots, x_{m-1}, x_{m}$ are affinely independent. Hence $\mathcal{O}$ is indeed an open set.
\end{proof}

\begin{fact}  
{\rm \cite[Theorem~9.26]{D2010}}
\label{fact:BestAppAffSubspace}
Let $V$ be an affine subset of $\mathcal{H}$, say
$V=M+v$, where $M$ is a linear subspace of $\mathcal{H}$ 
and $v\in V$. Let $x \in \mathcal{H}$ and $y_{0} \in \mathcal{H}$. 
Then the following statements are equivalent:
\begin{enumerate}
\item  \label{fact:BestAppAffSubspace:i} $y_{0}=P_{V}(x)$.
\item  \label{fact:BestAppAffSubspace:ii} $x-y_{0} \in M^{\perp}$.
\item  \label{fact:BestAppAffSubspace:iii} $\innp{x-y_{0}, y-v}=0 ~~~~\mbox{for all}~ y \in V$.
\end{enumerate}
Moreover, 
\begin{align*}
P_{V}(x+e)=P_{V}(x)  ~~~~~~~~\mbox{for all} ~x \in X, e \in M^{\perp}.
\end{align*}
\end{fact}

\subsection{The Gram matrix}

\begin{definition} \label{defn:GramMatrix}
Let $ a_{1}, \ldots, a_{m} \in \mathcal{H}$. Then 
\begin{align*}
G(a_{1}, \ldots, a_{m})=
\begin{pmatrix} 
\norm{a_{1}}^{2} &\innp{a_{1},a_{2}} & \cdots & \innp{a_{1}, a_{m}}  \\ 
\innp{a_{2},a_{1}} & \norm{a_{2}}^{2} & \cdots & \innp{a_{2},a_{m}} \\
\vdots & \vdots & ~~& \vdots \\
\innp{a_{m},a_{1}} & \innp{a_{m},a_{2}} & \cdots & \norm{a_{m}}^{2} \\
\end{pmatrix} 
\end{align*}
is called the \emph{Gram matrix} of $a_{1}, \ldots, a_{m}$.
\end{definition}

\begin{fact} 
{\rm \cite[Theorem~6.5-1]{K1978}}
\label{fact:Gram:inver}  
Let $ a_{1}, \ldots, a_{m} \in \mathcal{H}$. Then the Gram matrix
$G(a_{1}, \ldots, a_{m})$ is invertible if and only if the vectors $a_{1}, \ldots, a_{m}$ are linearly independent.
\end{fact}

\begin{remark} \label{note:AffinIndpDetermNonZero}
Let $x,y,z$ be affinely independent vectors in $\mathbb{R}^{3}$. 
Set $a=y-x$ and $b=z-x$. Then, by \cref{fac:AffinIndeLineInd} and \cref{fact:Gram:inver}, $\norm{a}^{2} \norm{b}^{2}-\innp{a,b}^{2} \neq 0$ and $\norm{a} \neq 0$, $\norm{b} \neq 0$. 
\end{remark}

\begin{proposition} \label{prop:GramMatrSymm}
Let $x_{1}, \ldots, x_{m} \in \mathcal{H}$. Then for every $k \in
\{2, \ldots, m\}$, we have
\begin{align*}
\det \Big(G(x_{2}-x_{1}, \ldots, x_{m}-x_{1}) \Big) 
= \det \Big(G(x_{1}-x_{k}, \ldots,x_{k-1}-x_{k}, x_{k+1} -x_{k},\ldots, x_{m}-x_{k}) \Big)
\end{align*}
\end{proposition}

\begin{proof}
By \cref{defn:GramMatrix},  $G(x_{1}-x_{k}, \ldots,x_{k-1}-x_{k}, x_{k+1} -x_{k},\ldots, x_{m}-x_{k})$ is 
\begin{align}\label{eq:prop:Gram:x1k}
\begin{pmatrix}
\innp{x_{1}-x_{k},x_{1}-x_{k}}& \cdots & \innp{x_{1}-x_{k},x_{k-1}-x_{k}} & \innp{x_{1}-x_{k},x_{k+1}-x_{k}} &\cdots &\innp{x_{1}-x_{k},x_{m}-x_{k}} \\
\innp{x_{2}-x_{k},x_{1}-x_{k}}& \cdots & \innp{x_{2}-x_{k},x_{k-1}-x_{k}} & \innp{x_{2}-x_{k},x_{k+1}-x_{k}} &\cdots &\innp{x_{2}-x_{k},x_{m}-x_{k}} \\
\vdots & \cdots & \vdots & \vdots &\cdots &\vdots \\
\innp{x_{k-1}-x_{k},x_{1}-x_{k}}& \cdots & \innp{x_{k-1}-x_{k},x_{k-1}-x_{k}} & \innp{x_{k-1}-x_{k},x_{k+1}-x_{k}} &\cdots &\innp{x_{k-1}-x_{k},x_{m}-x_{k}} \\
\innp{x_{k+1}-x_{k},x_{1}-x_{k}}& \cdots & \innp{x_{k+1}-x_{k},x_{k-1}-x_{k}} & \innp{x_{k+1}-x_{k},x_{k+1}-x_{k}} &\cdots &\innp{x_{k+1}-x_{k},x_{m}-x_{k}} \\
\vdots & \cdots & \vdots & \vdots &\cdots &\vdots \\
\innp{x_{m}-x_{k},x_{1}-x_{k}}& \cdots & \innp{x_{m}-x_{k},x_{k-1}-x_{k}} & \innp{x_{m}-x_{k},x_{k+1}-x_{k}} &\cdots &\innp{x_{m}-x_{k},x_{m}-x_{k}} \\
\end{pmatrix}. 
\end{align}
In \cref{eq:prop:Gram:x1k}, we perform the following 
elementary row and column operations: 
For every $i \in \{2,3,\ldots, m-1\}$, 
subtract the $1^{\text{st}}$ row from the $i^{\text{th}}$ row,
and then subtract 
the $1^{\text{st}}$ column from the $i^\text{th}$ column.
Then multiply $1^{\text{st}}$ row and $1^{\text{st}}$ column by
$-1$, respectively. It follows that 
the determinant of \cref{eq:prop:Gram:x1k} equals the determinant of
\begin{align}\label{eq:prop:Gram:xk1}
\begin{pmatrix}
\innp{x_{k}-x_{1},x_{k}-x_{1}}& \cdots & \innp{x_{k}-x_{1},x_{k-1}-x_{1}} & \innp{x_{k}-x_{1},x_{k+1}-x_{1}} &\cdots &\innp{x_{k}-x_{1},x_{m}-x_{1}} \\
\innp{x_{2}-x_{1},x_{k}-x_{1}}& \cdots & \innp{x_{2}-x_{1},x_{k-1}-x_{1}} & \innp{x_{2}-x_{1},x_{k+1}-x_{1}} &\cdots &\innp{x_{2}-x_{1},x_{m}-x_{1}} \\
\vdots & \cdots & \vdots & \vdots &\cdots &\vdots \\
\innp{x_{k-1}-x_{1},x_{k}-x_{1}}& \cdots & \innp{x_{k-1}-x_{1},x_{k-1}-x_{1}} & \innp{x_{k-1}-x_{1},x_{k+1}-x_{1}} &\cdots &\innp{x_{k-1}-x_{1},x_{m}-x_{1}} \\
\innp{x_{k+1}-x_{1},x_{k}-x_{1}}& \cdots & \innp{x_{k+1}-x_{1},x_{k-1}-x_{1}} & \innp{x_{k+1}-x_{1},x_{k+1}-x_{1}} &\cdots &\innp{x_{k+1}-x_{1},x_{m}-x_{1}} \\
\vdots & \cdots & \vdots & \vdots &\cdots &\vdots \\
\innp{x_{m}-x_{1},x_{k}-x_{1}}& \cdots & \innp{x_{m}-x_{1},x_{k-1}-x_{1}} & \innp{x_{m}-x_{1},x_{k+1}-x_{1}} &\cdots &\innp{x_{m}-x_{1},x_{m}-x_{1}} \\
\end{pmatrix}.
\end{align}
In \cref{eq:prop:Gram:xk1}, we interchange $i^{\text{th}}$ row
and $(i+1)^{\text{th}}$ successively for $i=1, \ldots,k-2$. In
addition, we interchange $j^{\text{th}}$ column and
$(j+1)^{\text{th}}$ column successively for $j=1, \ldots,k-2$.
Then the resulting matrix is just $G(x_{2}-x_{1}, \ldots,
x_{m}-x_{1})$. 
Because the number of interchange we performed is even, the 
determinant is unchanged. Therefore, we obtain
 \begin{align*}
  \det \Big(G(x_{1}-x_{k}, \ldots,x_{k-1}-x_{k}, x_{k+1} -x_{k},\ldots, x_{m}-x_{k}) \Big) = \det \Big(G(x_{2}-x_{1}, \ldots, x_{m}-x_{1}) \Big)
 \end{align*}
 as claimed. 
\end{proof}

\begin{fact}
{\rm \cite[page 16]{T2008}} 
\label{fact:MatrixDeterInverse}
Let $S = \{ A \in \mathbb{R}^{n\times n} ~|~A ~\text{is invertible}~ \}$. Then the mapping $S \to S : A \mapsto A^{-1}$ is continuous.
\end{fact}

\begin{fact}[Cramer's rule]  
{\rm \cite[page 476]{MC2000}}
\label{fact:CramerRule}
If $A \in \mathbb{R}^{n\times n}$ is invertible and $Ax=b$, then for every $i \in \{1, \ldots,n\}$,
we have 
\begin{align*}
x_{i} = \frac{\det(A_{i})}{\det(A)},
\end{align*}
where $A_{i} =[ A_{*,1}|\cdots|A_{*,i-1}|b|A_{*,i+1}|\cdots|A_{*,n}]$. That is, $A_{i}$ is identical to $A$ except that column $A_{*,i}$ has been replaced by $b$.
\end{fact}

\begin{corollary} \label{cor:GramInver:Continu}
Let $\{x_{1}, \ldots, x_{m} \} \subseteq \mathcal{H}$ with $x_{1}, \ldots, x_{m} $ being affinely independent. Let $\big( (x^{(k)}_{1}, \ldots,x^{(k)}_{m})  \big)_{k \in \mathbb{N}} \subseteq \mathcal{H}^{m}$ such that
\begin{align*}
\lim_{k \rightarrow \infty} (x^{(k)}_{1}, \ldots,x^{(k)}_{m}) = (x_{1}, \ldots, x_{m}).
\end{align*}
Then 
\begin{align*}
G(x_{2}-x_{1}, \ldots, x_{m}-x_{1})^{-1}= \lim_{k \rightarrow \infty}G(x^{(k)}_{2}-x_{1}^{(k)}, \ldots,x^{(k)}_{m}-x_{1}^{(k)})^{-1}.
\end{align*}
\end{corollary}

\begin{proof}
By \cref{lem:AffineIndep:OpenSet}, we know there exists $K \in \mathbb{N}$ such that 
\begin{align*}
(\forall k \geq K) \quad x^{(k)}_{1}, \ldots, x^{(k)}_{m} ~\text{are affinely independent}.
\end{align*}
Using \cref{fac:AffinIndeLineInd}, we know 
\begin{align*}
x_{2}-x_{1}, \ldots, x_{m}-x_{1} ~\text{are linearly independent},
\end{align*}
and
\begin{align*}
(\forall k \geq K) \quad x^{(k)}_{2}-x^{(k)}_{1}, \ldots, x^{(k)}_{m}-x^{(k)}_{1}  ~\text{are linearly independent}.
\end{align*} 
Hence \cref{fact:Gram:inver} tells us that $G(x_{2}-x_{1}, \ldots, x_{m}-x_{1})^{-1}$ and $(\forall k \geq K)$ $G(x^{(k)}_{2}-x_{1}^{(k)}, \ldots,x^{(k)}_{m}-x_{1}^{(k)})^{-1}$  exist. Therefore, the required result follows directly from \cref{fact:MatrixDeterInverse}.
\end{proof}

\section{The circumcenter} 
\label{sec:DefinCircOper}

Before we are able to define the main actor in this paper, the
circumcenter operator, we shall require a few more more results. 

\begin{proposition} \label{prop:NormEqInnNorm}
Let $p,x,y \in \mathcal{H}$, and set 
$U=\aff\{x,y\}$. Then the following are equivalent:
\begin{enumerate} 
\item \label{prop:NormEqInnNorm:Norm} $\norm{p-x} =\norm{p-y}$.
\item \label{prop:NormEqInnNorm:NorInnp} $\innp{p-x,y-x} =\frac{1}{2} \norm{y-x}^{2}$.
\item \label{prop:NormEqInnNorm:ProjeEqu}$P_{U}(p)=\frac{x+y}{2}$.
\item \label{prop:NormEqInnNorm:ProjeBelo} $p \in \frac{x+y}{2} +(U-U)^{\perp}$.
\end{enumerate} 
\end{proposition}

\begin{proof}
It is clear that
\begin{align*}
\norm{p-x} =\norm{p-y}
&\Longleftrightarrow \norm{p-x}^{2} =\norm{(p-x)+(x-y)}^{2}\\
&\Longleftrightarrow \norm{p-x}^{2} =\norm{p-x}^{2} + 2\innp{p-x,x-y} +\norm{x-y}^{2}\\
&\Longleftrightarrow \innp{p-x,y-x} =\frac{1}{2} \norm{y-x}^{2}. 
\end{align*}
Hence we get \cref{prop:NormEqInnNorm:Norm} $\Leftrightarrow$ \cref{prop:NormEqInnNorm:NorInnp}.

Notice $\frac{x+y}{2} \in U$. Now
\begin{align*}
\frac{x+y}{2} = P_{U}(p) \Longleftrightarrow & (\forall u \in U) \quad \innp{p- \frac{x+y}{2}, u -x}=0 \quad (\text{by \cref{fact:BestAppAffSubspace:i}} \Leftrightarrow \text{\cref{fact:BestAppAffSubspace:iii} in \cref{fact:BestAppAffSubspace}})\\
\Longleftrightarrow &  (\forall \alpha \in \mathbb{R}) \quad \innp{p- \frac{x+y}{2}, (x+ \alpha(y-x)) -x}=0 \quad (\text{by}~U=x +\spn\{y-x\})\\
\Longleftrightarrow &  \innp{p- \frac{x+y}{2}, y-x}=0 \\
\Longleftrightarrow &  \innp{p-(x -\frac{x-y}{2}), y-x}=0 \\
\Longleftrightarrow &  \innp{p-x, y-x} + \innp{\frac{x-y}{2},y-x}=0 \\
\Longleftrightarrow & \innp{p-x,y-x} =\frac{1}{2} \norm{y-x}^{2},
\end{align*}
which imply that  \cref{prop:NormEqInnNorm:ProjeEqu} $\Leftrightarrow$  \cref{prop:NormEqInnNorm:NorInnp}.

On the other hand, by \cref{fact:BestAppAffSubspace:i} $\Leftrightarrow$ \cref{fact:BestAppAffSubspace:ii} in \cref{fact:BestAppAffSubspace} and by \cref{fac:AffinePointLinearSpace},
\begin{align*}
\frac{x+y}{2} = P_{U}(p) \Longleftrightarrow & p- \frac{x+y}{2} \in (U-U)^{\perp}\\
\Longleftrightarrow & p \in \frac{x+y}{2} + (U-U)^{\perp},
\end{align*}
which yield that \cref{prop:NormEqInnNorm:ProjeEqu} $\Leftrightarrow$ \cref{prop:NormEqInnNorm:ProjeBelo}.

In conclusion, we obtain \cref{prop:NormEqInnNorm:Norm} $\Leftrightarrow$ \cref{prop:NormEqInnNorm:NorInnp} $\Leftrightarrow$ \cref{prop:NormEqInnNorm:ProjeEqu} $\Leftrightarrow$ \cref{prop:NormEqInnNorm:ProjeBelo}.
\end{proof}

\begin{corollary}\label{cor:LongNormEqInnNorm}
Let $x_{1}, \ldots, x_{m}$ be in $\mathcal{H}$. Let $p \in \mathcal{H}$. Then
\begin{align*}
\norm{p-x_{1}}=\cdots =\norm{p-x_{m-1}}=\norm{p-x_{m}} \Longleftrightarrow \begin{cases}
\innp{p-x_{1},x_{2}-x_{1}} = \frac{1}{2} \norm{x_{2}-x_{1}}^{2} \\
~~~~~~~~~~\vdots \\
\innp{p-x_{1},x_{m-1}-x_{1}} = \frac{1}{2} \norm{x_{m-1}-x_{1}}^{2} \\
\innp{p-x_{1},x_{m}-x_{1}} = \frac{1}{2} \norm{x_{m}-x_{1}}^{2}.
\end{cases}
\end{align*}
\end{corollary}

\begin{proof}
Set $I=\{2, \ldots, m-1,m\}$, and let $i \in I$. In \cref{prop:NormEqInnNorm}, substitute $x=x_{1}$ and $y =x_{i}$ and use \cref{prop:NormEqInnNorm:Norm} $\Leftrightarrow$ \cref{prop:NormEqInnNorm:NorInnp}. Then we get $\norm{p-x_{1}} =\norm{p-x_{i}} \Longleftrightarrow \innp{p-x_{1},x_{i}-x_{1}} = \frac{1}{2} \norm{x_{i}-x_{1}}^{2}$. Hence
\begin{equation*}
(\forall i \in I) \quad \norm{p-x_{1}} =\norm{p-x_{i}} \Longleftrightarrow \innp{p-x_{1},x_{i}-x_{1}} = \frac{1}{2} \norm{x_{i}-x_{1}}^{2}.
\end{equation*}
Therefore, the proof is complete.
\end{proof}

The next result plays an essential role in the definition of the circumcenter operator.
\begin{proposition} \label{prop:unique:PExisUnique}
Set $S=\{x_{1}, x_{2}, \ldots, x_{m}\}$, where $m \in \mathbb{N} 
\smallsetminus \{0\}$ and $x_{1}, x_{2}, \ldots, x_{m}$ are in $\mathcal{H}$. 
Then there is at most one point $p \in \mathcal{H}$ satisfying the following two conditions:
\begin{enumerate}
\item \label{prop:unique:PExisUnique:i} $p \in \aff(S)$, and
\item  \label{prop:unique:PExisUnique:ii} $\{ \norm{p-s}~|~s \in S\}$ is a
singleton:
$\norm{p-x_{1}} =\norm{p-x_{2}}=\cdots =\norm{p-x_{m}}$.
\end{enumerate}
\end{proposition}

\begin{proof}
Assume both of $p, q$ satisfy conditions \cref{prop:unique:PExisUnique:i} and \cref{prop:unique:PExisUnique:ii}.

By assumption and \cref{lem:AffineHull}, $p,  q \in  \aff(S)=\aff \{x_{1}, \ldots, x_{m} \} =x_{1} +\spn \{ x_{2}-x_{1}, \ldots, x_{m}-x_{1} \} $. Thus $p-q \in \spn \{ x_{2}-x_{1}, \ldots, x_{m}-x_{1} \} $, and so there exist $ \alpha_{1}, \ldots , \alpha_{m-1} \in \mathbb{R}$ such that $ p-q= \sum^{m-1}_{i=1} \alpha_{i}(x_{i+1}-x_{1})$.
Using the \cref{cor:LongNormEqInnNorm} above and using the
condition \cref{prop:unique:PExisUnique:ii} satisfied by both of $p$ and $q$,
we observe that for every $i \in I=\{2, \ldots,m\}$, we have 
\begin{align*}
\innp{p-x_{1}, x_{i}-x_{1}}&=\frac{1}{2}\norm{x_{i}-x_{1}}^{2} \quad \text{and}\\
\innp{q-x_{1}, x_{i}-x_{1}}&=\frac{1}{2}\norm{x_{i}-x_{1}}^{2}.
\end{align*}
Subtracting the above equalities, we get 
\begin{align*}
(\forall i \in I) \quad \innp{p-q, x_{i}-x_{1}}=0.
\end{align*}
Multiplying $\alpha_{i}$ on both sides of the corresponding $i^{\text{th}}$ 
equality and then summing up the $m-1$ equalities, we get
\begin{align*}
0= \Innp{p-q,\sum^{m-1}_{i=1} \alpha_{i}(x_{i+1}-x_{1})} =  \innp{p-q,p-q}=\norm{p-q}^{2}.
\end{align*}
Hence $p=q$, which implies that if such point satisfying conditions \cref{prop:unique:PExisUnique:i} and \cref{prop:unique:PExisUnique:ii} exists, then it must be unique.
\end{proof}

We are now in a position to define the circumcenter operator.

\begin{definition}[circumcenter] \label{defn:Circumcenter}
The \emph{circumcenter operator} is 
\begin{empheq}[box=\mybluebox]{equation*}
\CCO{} \colon \mathcal{P}(\mathcal{H}) \to \mathcal{H} \cup \{ \varnothing \} \colon S \mapsto \begin{cases} p, \quad ~\text{if}~p \in \aff (S)~\text{and}~\{\norm{p-s} ~|~s \in S \}~\text{is a singleton};\\
\varnothing, \quad~ \text{otherwise}.
\end{cases}
\end{empheq}
The \emph{circumradius operator} is 
\begin{align*}
\CRO{} \colon \mathcal{P}(\mathcal{H}) \to \mathbb{R}\colon
S \mapsto 
\begin{cases} \norm{\CCO(S) -s }, &\text{if $\CCO(S) \in
\mathcal{H}$ and $s\in S$};\\
+\infty, &\text{if $\CCO(S)=\varnothing$.}
\end{cases}
\end{align*}
In particular, when $\CCO(S) \in \mathcal{H}$, that is, $\CCO(S) \neq \varnothing$, we say that the circumcenter of $S$ exists and we call $\CCO(S)$ the circumcenter of $S$ and $\CRO(S)$ the circumradius of $S$.
\end{definition}

Note that in the \cref{prop:unique:PExisUnique} above, we have already proved that for every $S \in \mathcal{P}(\mathcal{H})$, there is at most one point $p \in \aff (S) $ such that $\{\norm{p-s} ~|~s \in S \}$ is a singleton, so the notions are \emph{well defined}. 
Hence we obtain the following alternative expression of the
circumcenter operator: 

\begin{remark} \label{rem:Circumcenter}
Let $S \in \mathcal{P}(\mathcal{H})$. Then the $\CCO(S)$ is either $\varnothing$ or the \emph{unique} point $p \in \mathcal{H}$ such that 
\begin{enumerate}
\item $p \in \aff (S)$ and, 
\item $\{\norm{p-s}~|~s \in S \}$ is a singleton.
\end{enumerate}
\end{remark}

\begin{example} \label{exam:CircForTwoPoints}
Let $x_1,x_2$ be in $\mathcal{H}$. 
Then 
\begin{align*}
\CCO{\big(\{x_1,x_2\}\big)}=\frac{x_{1} + x_{2}}{2}.
\end{align*} 
\end{example}
 
 \section{Explicit formulae for the circumcenter} 
\label{sec:ClosFormuCircOper}

We continue to assume that 
\begin{empheq}[box=\mybluebox]{equation*}
m \in \mathbb{N} \smallsetminus \{0\},
\quad x_1,\ldots,x_m \text{~are vectors in $\mathcal{H}$},\quad 
\text{and} \quad S = \{x_{1}, \ldots, x_{m}\}. 
\end{empheq}
If $S$ is a singleton, say $S=\{x_{1}\}$, 
then, by \cref{defn:Circumcenter}, we clearly have $\CCO(S)=x_{1}$. 
So in this section, to deduce the formula of $\CCO(S)$, we always 
assume that 
\begin{empheq}[box=\mybluebox]{equation*}
m \geq 2.
\end{empheq}

We are ready for an explicit formula for the circumcenter.

\begin{theorem} \label{thm:unique:LinIndpPformula}
Suppose that $x_{1}, \ldots, x_{m}$ are affinely independent.
Then $\CCO(S) \in \mathcal{H}$, which means that $\CCO(S)$ 
is the unique point satisfying the following two conditions:
\begin{enumerate}
\item \label{prop:unique:i} $\CCO(S) \in \aff (S)$, and
\item  \label{prop:unique:ii} $\{ \norm{\CCO(S)-s}~|~s \in S \}$ is a singleton.
\end{enumerate}
Moreover,
\begin{align*}
\CCO(S)= x_{1}+\frac{1}{2}(x_{2}-x_{1},\ldots,x_{m}-x_{1})
 G( x_{2}-x_{1},\ldots,x_{m}-x_{1})^{-1}
\begin{pmatrix}
 \norm{x_{2}-x_{1}}^{2} \\
 \vdots\\
\norm{x_{m}-x_{1}}^{2} \\
\end{pmatrix},
\end{align*}
where $G( x_{2}-x_{1},\ldots,x_{m-1}-x_{1},x_{m}-x_{1})$ is the
Gram matrix defined in \cref{defn:GramMatrix}:
\begin{align*}
&G( x_{2}-x_{1},\ldots, x_{m-1}-x_{1},x_{m}-x_{1})\\
=&
\begin{pmatrix} 
\norm{x_{2}-x_{1}}^{2} &\innp{x_{2}-x_{1},x_{3}-x_{1}} & \cdots & \innp{x_{2}-x_{1}, x_{m}-x_{1}}  \\ 
\vdots & \vdots & ~~& \vdots \\
\innp{x_{m-1}-x_{1},x_{2}-x_{1}} & \innp{x_{m-1}-x_{1}, x_{3}-x_{1}} & \cdots & \innp{x_{m-1}-x_{1},x_{m}-x_{1}} \\
\innp{x_{m}-x_{1},x_{2}-x_{1}} & \innp{x_{m}-x_{1},x_{3}-x_{1}} & \cdots & \norm{x_{m}-x_{1}}^{2} \\
\end{pmatrix}.
\end{align*}
\end{theorem}

\begin{proof}
By assumption and \cref{fac:AffinIndeLineInd}, we get that
$x_{2}-x_{1}, \ldots, x_{m}-x_{1} $ are linearly independent.
Then by \cref{fact:Gram:inver}, the Gram matrix $G(x_{2}-x_{1},
x_{3}-x_{1}, \ldots, x_{m}-x_{1})$ is invertible. Set
\begin{align*}
\begin{pmatrix}
\alpha_{1} \\
\alpha_{2}\\
\vdots\\
\alpha_{m-1} \\
\end{pmatrix}
= \frac{1}{2}G(x_{2}-x_{1}, x_{3}-x_{1}, \ldots, x_{m}-x_{1})^{-1}
\begin{pmatrix}
 \norm{x_{2}-x_{1}}^{2} \\
  \norm{x_{3}-x_{1}}^{2} \\
\vdots\\
 \norm{x_{m}-x_{1}}^{2} \\
\end{pmatrix},
\end{align*}
and 
\begin{align*}
p = x_{1}+\alpha_{1}(x_{2}-x_{1})+\alpha_{2}(x_{3}-x_{1})+\cdots+\alpha_{m-1}(x_{m}-x_{1}).
\end{align*}
By the definition of $G(x_{2}-x_{1}, x_{3}-x_{1}, \ldots,
x_{m}-x_{1})$ and by the definitions of $\begin{pmatrix}
\alpha_{1} &
\alpha_{2}&
\cdots&
\alpha_{m-1} 
\end{pmatrix}^{\intercal}$ and $p$,  we obtain
the equivalences
\begin{align*}
&G(x_{2}-x_{1}, x_{3}-x_{1}, \ldots, x_{m}-x_{1})
\begin{pmatrix}
\alpha_{1} \\
\alpha_{2}\\
\vdots\\
\alpha_{m-1} \\
\end{pmatrix}
= \frac{1}{2} 
\begin{pmatrix}
\norm{x_{2}-x_{1}}^{2} \\
 \norm{x_{3}-x_{1}}^{2} \\
\vdots\\
 \norm{x_{m}-x_{1}}^{2} \\
\end{pmatrix} \\
\Longleftrightarrow &
\begin{cases}
\innp{ \alpha_{1}(x_{2}-x_{1})+ \cdots +\alpha_{m-1}(x_{m}-x_{1}), x_{2}-x_{1} } = \frac{1}{2} \norm{x_{2}-x_{1}}^{2} \\
~~~~~~~~~~\vdots \\
\innp{\alpha_{1}(x_{2}-x_{1})+ \cdots +\alpha_{m-1}(x_{m}-x_{1}), x_{m}-x_{1} } = \frac{1}{2} \norm{x_{m}-x_{1}}^{2}
\end{cases}\\
\Longleftrightarrow &
\begin{cases}
\innp{p-x_{1},x_{2}-x_{1}} = \frac{1}{2} \norm{x_{2}-x_{1}}^{2} \\
~~~~~~~~~~\vdots \\
\innp{p-x_{1},x_{m}-x_{1}} = \frac{1}{2} \norm{x_{m}-x_{1}}^{2}.
\end{cases}
\end{align*}
Hence by \cref{cor:LongNormEqInnNorm}, we know that $p$ satisfy condition \cref{prop:unique:ii}. In addition, it is clear that $p=x_{1}+\alpha_{1}(x_{2}-x_{1})+\alpha_{2}(x_{3}-x_{1})+\cdots+\alpha_{m-1}(x_{m}-x_{1}) \in x_{1}+ \spn \{x_{2} -x_{1}, \ldots, x_{m}-x_{1} \}= \aff(S)$, which is just the condition \cref{prop:unique:i}. Hence the point satisfying conditions  \cref{thm:unique:i} and \cref{thm:unique:ii} exists.

Moreover, by \cref{prop:unique:PExisUnique}, if the point exists, then it must be unique. 
\end{proof}

\begin{lemma}\label{lem:unique:BasisPformula}
Suppose that $\CCO(S) \in \mathcal{H}$, 
and let $K \subseteq S$ such that $\aff(K)=\aff(S)$. 
Then
\begin{align*}
\CCO(S)=\CCO(K).
\end{align*}
\end{lemma}

\begin{proof}  
By assumption, $\CCO(S) \in \mathcal{H}$, that is: 
\begin{enumerate}
\item \label{thm:unique:i} $\CCO(S) \in \aff (S)$, and
\item  \label{thm:unique:ii} $\{ \norm{\CCO(S)-s}~|~s \in S \}$ is a singleton.
\end{enumerate}

Because $K \subseteq S$, we get $\{ \norm{\CCO(S)-s}~|~s \in K \}$ is a singleton, by \cref{thm:unique:ii}. Since $\aff(K)=\aff(S)$, by \cref{thm:unique:i}, the point $\CCO(S)$ satisfy
\begin{enumerate}[label=(\Roman*)]
\item $\CCO(S) \in \aff (K)$, and
\item $\{ \norm{\CCO(S) - u}~|~ u \in K\}$ is a singleton.
\end{enumerate}
Replacing $S$ in \cref{prop:unique:PExisUnique} by $K$ and combining with \cref{defn:Circumcenter}, we know $\CCO(K) =\CCO(S)$.  
 \end{proof}

\begin{corollary} \label{cor:unique:BasisPformula}
Suppose that $\CCO(S) \in \mathcal{H}$. Let $x_{i_{1}}, \ldots,
x_{i_{t}}$ be elements of $S$ such that $x_{1}, x_{i_{1}},
\ldots, x_{i_{t}}$ are affinely independent, and 
set $K=\{x_{1}, x_{i_{1}}, \ldots, x_{i_{t}} \}$. 
Furthermore, assume that $ \aff (K) =\aff(S)$. %\footnotemark ~  
Then 
\begin{align*}
\CCO(S)=\CCO(K) = x_{1}+\frac{1}{2}(x_{i_{1}}-x_{1},\ldots,x_{i_{t}}-x_{1})
 G( x_{i_{1}}-x_{1},\ldots,x_{i_{t}}-x_{1})^{-1}
\begin{pmatrix}
 \norm{x_{i_{1}}-x_{1}}^{2} \\
 \vdots\\
 \norm{x_{i_{t}}-x_{1}}^{2}\\ 
\end{pmatrix}.
\end{align*}
\end{corollary}

\begin{proof}
By \cref{thm:unique:LinIndpPformula}, $x_{1}, x_{i_{1}}, \ldots, x_{i_{t}}$ are affinely independent implies that $\CCO(K) \neq \varnothing$, and 
\begin{align*}
\CCO(K)= x_{1}+\frac{1}{2}(x_{i_{1}}-x_{1},\ldots,x_{i_{t}}-x_{1})
 G( x_{i_{1}}-x_{1},\ldots,x_{i_{t}}-x_{1})^{-1}
\begin{pmatrix}
 \norm{x_{i_{1}}-x_{1}}^{2} \\
 \vdots\\
 \norm{x_{i_{t}}-x_{1}}^{2}\\ 
\end{pmatrix}.
\end{align*}
Then the desired result follows from \cref{lem:unique:BasisPformula}.
\end{proof}

\begin{lemma}\label{lem:Basis:AffineHullEq} 
Let $x_{i_{1}}, \ldots, x_{i_{t}}$ be elements of $S$, 
and set $K =\{x_{1}, x_{i_{1}}, \ldots, x_{i_{t}} \}$. Then 
\begin{align*}
& \aff(K) =\aff(S)~\mbox{ and}~ x_{1}, x_{i_{1}}, \ldots, x_{i_{t}}~\mbox{are affinely independent}.\\
\Longleftrightarrow ~& x_{i_{1}}-x_{1}, \ldots, x_{i_{t}}-x_{1} ~\mbox{is a basis of}~ \spn\{x_{2}-x_{1}, \ldots, x_{m}-x_{1} \}
\end{align*}
\end{lemma}

\begin{proof}
Indeed, 
\begin{align*}
&~~ x_{i_{1}}-x_{1}, \ldots, x_{i_{t}}-x_{1}~\mbox{is a basis of}~ \spn\{x_{2}-x_{1}, \ldots, x_{m}-x_{1} \}\\
\Longleftrightarrow & \begin{cases} x_{i_{1}}-x_{1}, \ldots, x_{i_{t}}-x_{1} ~\mbox{are linearly independent, and}\\
\spn \{x_{i_{1}}-x_{1}, \ldots, x_{i_{t}}-x_{1} \} = \spn\{x_{2}-x_{1}, \ldots, x_{m}-x_{1} \}
\end{cases}\\
\stackrel{\text{\cref{fac:AffinIndeLineInd}}}{\Longleftrightarrow} & \begin{cases}x_{1}, x_{i_{1}}, \ldots, x_{i_{t}}~\mbox{are affinely independent, and}\\
x_{1}+\spn \{x_{i_{1}}-x_{1}, \ldots, x_{i_{t}}-x_{1} \} = x_{1}+\spn\{x_{2}-x_{1}, \ldots, x_{m}-x_{1} \}
\end{cases}\\
\Longleftrightarrow & \begin{cases}x_{1}, x_{i_{1}}, \ldots, x_{i_{t}}~\mbox{are affinely independent, and}\\
 \aff(K)= \aff(S),
\end{cases}
\end{align*}
which completes the proof. 
\end{proof}

\iffalse
\cref{cor:unique:BasisPformula} and \cref{lem:Basis:AffineHullEq}  tell us that if only $\CCO(S) \in \mathcal{H}$, then we can use the programming software to calculate it easily.
\fi
 
\section{Additional formulae for the circumcenter} 
\label{sec:SymmFormuCCS}

Upholding the assumptions of \cref{sec:ClosFormuCircOper},
we assume additionally that 
\begin{empheq}[box=\mybluebox]{equation*} 
x_{1}, \ldots, x_{m}  ~\text{are affinely independent.}
\end{empheq}
By \cref{thm:unique:LinIndpPformula}, $\CCO(S) \in \mathcal{H}$. Let 
\begin{empheq}[box=\mybluebox]{equation*}
k \in \{2, 3, \ldots, m\}~\text{be arbitrary but fixed}.
\end{empheq}
By \cref{thm:unique:LinIndpPformula} again, we know that 
\begin{subequations} \label{eq:FormuCCSAlpha}
\begin{align} 
\CCO(S) &~=  x_{1}+\alpha_{1}(x_{2}-x_{1})+\alpha_{2}(x_{3}-x_{1})+\cdots+\alpha_{m-1}(x_{m}-x_{1}) \label{eq:FormuSymmCCS:1:F}\\
&~= \big(1 - {\textstyle \sum^{m-1}_{i=1}} \alpha_{i}\big) x_{1} + \alpha_{1} x_{2} +\cdots+\alpha_{m-1} x_{m}, \label{eq:FormuSymmCCS:1}
\end{align}
\end{subequations}
where 
\begin{align} \label{eq:FormCCS:Param:Alpha}
\begin{pmatrix}
\alpha_{1} \\
\alpha_{2}\\
\vdots\\
\alpha_{m-1} \\
\end{pmatrix}
= \frac{1}{2}G(x_{2}-x_{1}, x_{3}-x_{1}, \ldots, x_{m}-x_{1})^{-1}
\begin{pmatrix}
 \norm{x_{2}-x_{1}}^{2} \\
  \norm{x_{3}-x_{1}}^{2} \\
\vdots\\
 \norm{x_{m}-x_{1}}^{2} \\
\end{pmatrix}.
\end{align}
By the symmetry of the positions of the points $x_{1}, \ldots, x_{k}, \ldots,x_{m}$ in $S$ in \cref{defn:Circumcenter} and by \cref{prop:unique:PExisUnique}, we also get 
\begin{subequations}\label{eq:FormuCCSBeta}
\begin{align} 
\CCO(S) &~=  x_{k} +\beta_{1}(x_{1}-x_{k})+\cdots+\beta_{k-1}(x_{k-1}-x_{k})+\beta_{k}(x_{k+1}-x_{k})\cdots+\beta_{m-1}(x_{m}-x_{k}) \label{eq:FormuSymmCCS:k:F}\\ 
&~= \beta_{1} x_{1} +\cdots+\beta_{k-1}x_{k-1}+(1 - \sum^{m-1}_{i=1} \beta_{i}) x_{k} + \beta_{k}x_{k+1}+\cdots+\beta_{m-1} x_{m} , \label{eq:ForSyCCS:k}
\end{align}
\end{subequations}
where 
\begin{align} \label{eq:FormCCS:Param:Beta}
\begin{pmatrix}
\beta_{1} \\
\beta_{2}\\
\vdots\\
\beta_{m-1} \\
\end{pmatrix}
= \frac{1}{2}G(x_{1}-x_{k},\ldots,x_{k-1}-x_{k}, x_{k+1}-x_{k}, \ldots,x_{m}-x_{k})^{-1}
\begin{pmatrix}
 \norm{x_{1}-x_{k}}^{2} \\
 \vdots\\
 \norm{x_{k-1}-x_{k}}^{2}\\
 \norm{x_{k+1}-x_{k}}^{2}\\
 \vdots\\
\norm{x_{m}-x_{k}}^{2} \\
\end{pmatrix}.
\end{align}

\begin{proposition}
The following equalities hold:
\begin{align}
& \big(1 - \textstyle \sum^{m-1}_{i=1} \alpha_{i}\big) = \beta_{1},
~~~(\text{coefficient of $x_{1}$})  \label{eq:SFCCS} \\
& \alpha_{k-1}=\big(1 - \textstyle \sum^{m-1}_{i=1} \beta_{i}\big),
~~~(\text{coefficient of $x_{k}$})  \label{eq:SymForCCS:k} \\
&  (\forall i \in \{ 2, \ldots, k-1\}) \quad \alpha_{i-1} = \beta_{i} \quad \text{and} \quad (\forall j \in \{k, k+1, \ldots, m-1\}) \quad \alpha_{j} =\beta_{j}. \label{eq:SymForCCS:Remaining}
\end{align}
\end{proposition}

\begin{proof}
Recall that at the beginning of this section we assumed $x_{1}, \ldots, x_{m}$ are affinely independent. Combining the equations \cref{eq:FormuSymmCCS:1} $\&$\cref{eq:ForSyCCS:k} and \cref{lem:UniqExpreAffIndp}, we get the required results.
\end{proof}

To simplify the statements, we use the following abbreviations. 
\begin{align*}
A = ~&G(x_{2}-x_{1}, \ldots, x_{k}-x_{1}, \ldots, x_{m}-x_{1}),\\
B = ~& G(x_{1}-x_{k}, \ldots, x_{k-1}-x_{k}, x_{k+1}-x_{k}, \ldots, x_{m}-x_{k}),
\end{align*}
and the determinant of matrix $A$ (by \cref{prop:GramMatrSymm},
it is also the determinant of matrix $B$) is denoted by:
\begin{align*}
\delta = \det(A).
\end{align*}
We denote the two column vectors $a$, $b$ respectively by:
\begin{align*}
&~a =\begin{pmatrix} \norm{x_{2}-x_{1}}^{2}& \cdots&\norm{x_{k}-x_{1}}^{2}&\cdots&\norm{x_{m}-x_{1}}^{2} \end{pmatrix}^{\intercal},\\
&~b =\begin{pmatrix}
\norm{x_{1}-x_{k}}^{2}&\cdots&\norm{x_{k-1}-x_{k}}^{2}&\norm{x_{k+1}-x_{k}}^{2}&\cdots & \norm{x_{m}-x_{k}}^{2}
\end{pmatrix}^{\intercal}.
\end{align*}
For every $M \in \mathbb{R}^{n \times n}$, and for every $j \in \{1, 2, \ldots, n\}$,
\begin{empheq}[box=\mybluebox]{equation*}
\text{we denote the $j^{\text{th}}$ column of the matrix $M$ as $M_{*,j}$}.
\end{empheq}
In turn, for every $i \in \{1, \ldots, m-1\}$,
\begin{align*}
A_{i} =[A_{*,1}|\cdots|A_{*,i-1}|a|A_{*,i+1}|\cdots|A_{*,m-1}],
\end{align*}
and
\begin{align*}
B_{i} =[B_{*,1}|\cdots|B_{*,i-1}|b|B_{*,i+1}|\cdots|B_{*,m-1}].
\end{align*}
That is, $A_{i}$ is identical to $A$ except that column $A_{*,i}$ has been replaced by $a$ and $B_{i}$ is identical to $B$ except that column $B_{*,i}$ has been replaced by $b$.

\begin{lemma} \label{lem:CCSFormulasAlBe} The following statements hold:
\begin{enumerate}
\item \label{lem:SymFormCCSFormulasAl} $\begin{pmatrix} \alpha_{1} \cdots \alpha_{m-1} \end{pmatrix}^{\intercal}$ defined in \cref{eq:FormCCS:Param:Alpha} is the unique solution of the nonsingular system $Ay=\frac{1}{2}a$ where $y$ is the unknown variable. In consequence, for every $i \in \{1, \ldots, m-1\}$,
\begin{align*}
\alpha_{i}=\frac{\det(A_{i})}{2\delta}.
\end{align*}
\item \label{lem:SymFormCCSFormulasBe}
$\begin{pmatrix} \beta_{1} \cdots \beta_{m-1} \end{pmatrix}^{\intercal}$ defined in \cref{eq:FormCCS:Param:Beta} is the unique solution of the nonsingular system
$By=\frac{1}{2}b$ where $y$ is the unknown variable. In consequence, for every $i \in \{1, \ldots, m-1\}$,
\begin{align*}
\beta_{i}=\frac{\det(B_{i})}{2\delta}.
\end{align*}
\end{enumerate}
\end{lemma}

\begin{proof}
By assumption, $x_{1}, \ldots, x_{m}$ are affinely independent, and by \cref{prop:GramMatrSymm}, we know $\det(B)=\det(A)=\delta \neq 0$.

\cref{lem:SymFormCCSFormulasAl}: By definition of $\begin{pmatrix} \alpha_{1} \cdots \alpha_{m-1} \end{pmatrix}^{\intercal}$, 
\begin{align*}
\begin{pmatrix} \alpha_{1} \cdots \alpha_{m-1} \end{pmatrix}^{\intercal} = \frac{1}{2}A^{-1}a.
\end{align*}
Clearly we know it is the unique solution of the nonsingular system $Ay=\frac{1}{2}a$.  Hence the desired result follows directly from the \cref{fact:CramerRule}, the Cramer Rule.

\cref{lem:SymFormCCSFormulasBe}: Using the same method of proof of \cref{lem:SymFormCCSFormulasAl}, we can prove \cref{lem:SymFormCCSFormulasBe} .
\end{proof}

Using \cref{thm:unique:LinIndpPformula},
\cref{lem:CCSFormulasAlBe} and the equalities \cref{eq:SFCCS},
\cref{eq:SymForCCS:k} and \cref{eq:SymForCCS:Remaining}, we
readily obtain the following result. 

\begin{corollary}\label{cor:DifRepreCCS}  
Suppose that $x_{1}, \ldots, x_{m}$ are affinely independent. Then
\begin{align*}
\CCO(S) =\big(1- \textstyle \sum^{m-1}_{i=1} \alpha_{i}\big) 
x_{1} + \alpha_{1} x_{2} + \cdots+\alpha_{m-1} x_{m},
\end{align*}
where $(\forall i \in \{1, \ldots, m-1\})$ 
$\alpha_{i}= \frac{1}{2\delta} \det(A_{i})$.
Moreover, 
\begin{align*}
 1- \sum^{m-1}_{i=1} \alpha_{i}  =\frac{1}{2\delta}
 \det(B_{1}),~~~\alpha_{k-1}= 1- \sum^{m-1}_{i=1}
 \frac{1}{2\delta} \det(B_{i}),
 \end{align*}
 \begin{align*}
 (\forall i \in \{2, \ldots, k-1\}) \quad \alpha_{i -1}=\frac{1}{2\delta} 
 \det(B_{i}) \quad \text{and} \quad (\forall j \in \{k, k+1,
 \ldots, m-1\}) \quad \alpha_{j} =\frac{1}{2\delta} \det(B_{j}).
 \end{align*}
\end{corollary}

\section{Basic properties of the circumcenter} 
\label{sec:BasiPropCircOper}

In this section  we collect some fundamental 
properties of the circumcenter operator. 
Recall that 
\begin{empheq}[box=\mybluebox]{equation*}
m \in \mathbb{N} \smallsetminus \{0\},
\quad x_1,\ldots,x_m \text{~are vectors in $\mathcal{H}$},\quad 
\text{and} \quad S = \{x_{1}, \ldots, x_{m}\}. 
\end{empheq}

\begin{proposition}[scalar multiples]
\label{prop:CircumHomoge}
Let $\lambda \in \mathbb{R} \smallsetminus \{0\}$. 
Then 
$\CCO(\lambda S)=\lambda \CCO(S)$. 
\end{proposition}

\begin{proof}
Let $p \in \mathcal{H}$. 
By \cref{defn:Circumcenter}, 
\begin{align*}
p = \CCO(S) & \Longleftrightarrow \begin{cases}p \in \aff(S)\\ \{\norm{p-s}~|~s \in S \}~\text{is a singleton}~ \end{cases}\\
& \Longleftrightarrow \begin{cases}\lambda p \in \aff(\lambda S)\\ \{\norm{\lambda p-\lambda s}~|~\lambda s \in \lambda S \}~\text{is a singleton}~ \end{cases}\\
& \Longleftrightarrow p = \CCO(\lambda S),
\end{align*} 
and the result follows. 
\end{proof}

The next example below illustrates that 
we had to exclude the case $\lambda =0$ in \cref{prop:CircumHomoge}.
\begin{example}
Suppose that $\mathcal{H} =\mathbb{R}$ and that 
$S = \{0, -1, 1\}$. Then 
\begin{align*}
 \CCO(0\cdot S) =  \{0\}  \neq \varnothing  = 0\cdot \CCO(S).
\end{align*}
\end{example}

\begin{proposition}[translations]
\label{prop:CircumSubaddi}
Let $y \in \mathcal{H}$. Then
$\CCO(S+y)=\CCO(S)+y$.
\end{proposition}

\begin{proof}
Let $p \in \mathcal{H}$. By \cref{lem:AffineHull},
\begin{align*}
p \in \aff \{x_{1}, \ldots, x_{m}\} 
& \Longleftrightarrow (\exists~ \lambda_{1}, \ldots,\lambda_{m} \in \mathbb{R} ~\text{with}~\sum^{m}_{i=1} \lambda_{i}=1) \quad p=\sum^{m}_{i=1} \lambda_{i} x_{i}\\
& \Longleftrightarrow (\exists~ \lambda_{1}, \ldots,\lambda_{m} \in \mathbb{R} ~\text{with}~\sum^{m}_{i=1} \lambda_{i}=1) \quad p+y=\sum^{m}_{i=1} \lambda_{i} (x_{i}+y) \\
& \Longleftrightarrow p+y \in \aff \{x_{1}+y, \ldots, x_{m}+y\},
\end{align*}
that is
\begin{align} \label{eq:prop:CircumSubaddi}
p \in \aff (S) \Longleftrightarrow p+y \in \aff (S +y).
\end{align}
By \cref{eq:prop:CircumSubaddi} and \cref{rem:Circumcenter}, we have
\begin{align*}
p = \CCO(S) \in \mathcal{H}
& \Longleftrightarrow  \begin{cases}p \in \aff(S)\\ \{\norm{p-s}~|~s \in S \}~\text{is a singleton}~ \end{cases}\\
& \Longleftrightarrow  \begin{cases}p+y \in \aff(S+y)\\ \{\norm{(p+y)-(s+y)}~|~s+y \in S+y \}~\text{is a singleton}~ \end{cases}\\
& \Longleftrightarrow p+y = \CCO(S+y) \in \mathcal{H}.
\end{align*} 
Moreover, because $\varnothing =\varnothing +y$, the proof is
complete. 
\end{proof}

 \section{Circumcenters of sequences of sets} 
\label{sec:LimiCircOperSeqSet}
We uphold the assumption that 
\begin{empheq}[box=\mybluebox]{equation*}
m \in \mathbb{N} \smallsetminus \{0\},
\quad x_1,\ldots,x_m \text{~are vectors in $\mathcal{H}$},\quad 
\text{and} \quad S = \{x_{1}, \ldots, x_{m}\}. 
\end{empheq}

In this section, we explore the convergence of the circumcenter operator over a sequence of sets.

\begin{theorem} \label{thm:CCO:LimitCont}
Suppose that $\CCO(S) \in \mathcal{H}$. 
Then the following hold:
\begin{enumerate}
\item \label{prop:CCO:LimitCont:Linear} 
Set $t=\dim \Big( \spn\{x_{2}-x_{1}, \ldots, x_{m}-x_{1}\} \Big)$, 
and let $\widetilde{S}=\{x_{1}, x_{i_{1}}, \ldots, x_{i_{t}}\}
\subseteq S$ be such that $x_{i_{1}} -x_{1}, \ldots, x_{i_{t}}-x_{1}$ is a basis of $\spn\{x_{2}-x_{1}, \ldots, x_{m}-x_{1}\}$. 
Furthermore, let $\Big( (x^{(k)}_{1}, x^{(k)}_{i_{1}}, \ldots,
x^{(k)}_{i_{t}}) \Big)_{k \geq 1}$ $\subseteq$
$\mathcal{H}^{t+1}$ with $\lim_{k\rightarrow \infty}(
x^{(k)}_{1}, x^{(k)}_{i_{1}}, \ldots, x^{(k)}_{i_{t}})=(x_{1},
x_{i_{1}}, \ldots, x_{i_{t}})$, and 
set $(\forall k \geq 1)$ $\widetilde{S}^{(k)} = \{x^{(k)}_{1}, x^{(k)}_{i_{1}}, \ldots, x^{(k)}_{i_{t}}\}$. Then there exist $N \in \mathbb{N}$ such that for every $k \geq N$, $\CCO(\widetilde{S}^{(k)}) \in \mathcal{H}$ and
\begin{align*}
\lim_{k \rightarrow \infty} \CCO(\widetilde{S}^{(k)})= \CCO(\widetilde{S})=\CCO(S).
\end{align*}
\item  \label{prop:CCO:LimitCont:LinearIndep} 
Suppose that $ x_{1}, \ldots, x_{m-1}, x_{m}$ 
are affinely independent, 
and let $ \Big( (x^{(k)}_{1}, \ldots, x^{(k)}_{m-1}, x^{(k)}_{m})
\Big)_{k \geq 1}$ $\subseteq$  $\mathcal{H}^{m} $ satisfy
$\lim_{k\rightarrow \infty}( x^{(k)}_{1}, \ldots,x^{(k)}_{m-1},
x^{(k)}_{m})=(x_{1}, \ldots, x_{m-1},x_{m})$. Set $(\forall k \geq 1)$ $S^{(k)}=\{x^{(k)}_{1}, \ldots, x^{(k)}_{m-1}, x^{(k)}_{m}\}$. Then 
\begin{align*}
\lim_{k \rightarrow \infty} \CCO( S^{(k)} )= \CCO( S ).
\end{align*}
\end{enumerate}
\end{theorem}

\begin{proof}
\cref{prop:CCO:LimitCont:Linear}: Let $l$ be the cardinality of
the set $S$. Assume first that $l=1$. Then $t=0$, and $\widetilde{S}=\{x_{1}\}$. Let $ (x^{(k)}_{1} )_{k \geq 1} \subseteq \mathcal{H}$ satisfy $\lim_{k \rightarrow \infty} x^{(k)}_{1}=x_{1}$. By \cref{defn:Circumcenter}, we know $\CCO(\{ x^{(k)}_{1} \}) =x^{(k)}_{1}$ and $\CCO(\{ x_{1}\})=x_{1}$. Hence   
 \begin{align*}
\lim_{k \rightarrow \infty} \CCO(\widetilde{S}^{(k)})= \lim_{k \rightarrow \infty} x^{(k)}_{1}=x_{1}=\CCO(S).
\end{align*}
Now assume that $l \geq 2$. By \cref{cor:unique:BasisPformula}
and \cref{lem:Basis:AffineHullEq} , we obtain
\begin{align} \label{eq:prop:CCO:LimitCont:1}
\CCO(S)=\CCO(\widetilde{S})=x_{1}+\frac{1}{2}(x_{i_{1}}-x_{1},\ldots,x_{i_{t}}-x_{1})
 G( x_{i_{1}}-x_{1},\ldots,x_{i_{t}}-x_{1})^{-1}
\begin{pmatrix}
 \norm{x_{i_{1}}-x_{1}}^{2} \\
 \vdots\\
 \norm{x_{i_{t}}-x_{1}}^{2}\\ 
\end{pmatrix}.
\end{align}
Using the assumptions and the \cref{lem:AffineIndep:OpenSet}, we know 
that there exists $N \in \mathbb{N}$ such that
\begin{align*}
(\forall k \geq N) \quad x^{(k)}_{1}, x^{(k)}_{i_{1}}, \ldots, x^{(k)}_{i_{t}}~\mbox{are affinely independent.}
\end{align*} 
By \cref{thm:unique:LinIndpPformula}, we know $(k \geq N)$ $\CCO(\widetilde{S}^{(k)}) \in \mathcal{H}$. Moreover, for every $k \geq N$,
\begin{align}  \label{eq:prop:CCO:LimitCont:1k}
\CCO(\widetilde{S}^{(k)})=x^{(k)}_{1}+
\frac{1}{2}(x^{(k)}_{i_{1}}-x_{1}^{(k)} ,\ldots,x^{(k)}_{i_{t}}-x_{1}^{(k)})
 G( x^{(k)}_{i_{1}}-x_{1}^{(k)} ,\ldots,x^{(k)}_{i_{t}}-x_{1}^{(k)})^{-1}
\begin{pmatrix}
 \norm{x^{(k)}_{i_{1}}-x_{1}^{(k)}}^{2} \\
 \vdots\\
\norm{x^{(k)}_{i_{t}}-x_{1}^{(k)}}^{2} \\
\end{pmatrix}.
\end{align}
Comparing \cref{eq:prop:CCO:LimitCont:1} with \cref{eq:prop:CCO:LimitCont:1k} and using \cref{cor:GramInver:Continu}, we obtain
\begin{align*}
\lim_{k \rightarrow \infty} \CCO(\widetilde{S}^{(k)})= \CCO(\widetilde{S})=\CCO(S).
\end{align*}

\cref{prop:CCO:LimitCont:LinearIndep}: Let $ x_{1}, \ldots, x_{m-1}, x_{m} \in \mathcal{H}$ be affinely independent. Then $t = m-1$ and $\widetilde{S} =S$. Substitute the $\widetilde{S}$ and $\widetilde{S}^{(k)}$ in part \cref{prop:CCO:LimitCont:Linear} by our $S$ and  $S^{(k)}$ respectively. 
Then we obtain 
\begin{align*}
\lim_{k \rightarrow \infty} \CCO( S^{(k)} )= \CCO( S )
\end{align*}
and the proof is complete. 
\end{proof}

\begin{corollary} \label{cor:Psi:Contin}
The mapping 
\begin{align*}
\Psi\colon \mathcal{H}^{m} \rightarrow \mathcal{H} \cup \{ \varnothing \}
\colon 
(x_{1}, \ldots, x_{m}) \mapsto  \CCO(\{x_{1}, \ldots, x_{m}\}).
\end{align*}
is continuous at every point 
$(x_{1}, \ldots, x_{m})\in \mathcal{H}^{m}$ where 
$x_{1}, \ldots, x_{m}$ is affinely independent. 
\end{corollary}
\begin{proof}
This follows directly from \cref{thm:CCO:LimitCont}\cref{prop:CCO:LimitCont:LinearIndep}.
\end{proof}

Let us record the doubleton case explicitly.

\begin{proposition} \label{prop:CircumOperaContin2Points}
Suppose that $m=2$. 
Let $\big( ( x^{(k)}_{1},x^{(k)}_{2}) \big)_{k \geq 1} \subseteq
\mathcal{H}^{2}$ satisfy $\lim_{k \rightarrow \infty}( x^{(k)}_{1},x^{(k)}_{2})=(x_{1},x_{2})$. 
Then
\begin{align*}
\lim_{k \rightarrow \infty} \CCO(\{ x^{(k)}_{1},x^{(k)}_{2}\})=\CCO(\{x_{1},x_{2}\}).
\end{align*} 
\end{proposition}
\begin{proof}
Indeed, we deduce from \cref{exam:CircForTwoPoints} that 
\begin{align*}
\lim_{k \rightarrow \infty} \CCO( \{ x^{(k)}_{1},x^{(k)}_{2}\})=\lim_{k \rightarrow \infty} \frac{x^{(k)}_{1}+x^{(k)}_{2} }{2} =\frac{x_{1}+x_{2}}{2}=\CCO( \{x_{1},x_{2}\}) 
\end{align*}
and the result follows.
\end{proof}

The following example illustrates 
that the assumption that ``$m=2$'' in \cref{prop:CircumOperaContin2Points} 
cannot be replaced by ``the cardinality of $S$ is 2''.

\begin{example} \label{exam:CounterExampleContinuity:empty }
Suppose that 
$\mathcal{H}=\mathbb{R}^{2}$, that $m=3$, and that 
$S=\{x_{1}, x_{2},x_{3}\}$ with $x_{1}=(-1,0)$, $x_{2}=x_{3}=(1,0)$. Then there exists $( (x^{(k)}_{1}, x^{(k)}_{2},x^{(k)}_{3}) )_{k \geq 1} \subseteq \mathcal{H}^{3}$ such that
\begin{align*}
\lim_{k \rightarrow \infty} \CCO( \{ x^{(k)}_{1}, x^{(k)}_{2},x^{(k)}_{3}\}) \neq \CCO(S).
\end{align*}
\end{example}

\begin{proof}
For every $k \geq 1$, let $(x^{(k)}_{1}, x^{(k)}_{2},x^{(k)}_{3}) = \Big( (-1,0), (1,0), (1+\frac{1}{k},0) \Big) \in \mathcal{H}^{3}$. Then by \cref{defn:Circumcenter}, we know that $(\forall k \geq 1)$, $\CCO(S^{(k)})= \varnothing$, since there is no point in $\mathbb{R}^{2}$ which has equal distance to all of the three points.
On the other hand, by \cref{defn:Circumcenter} again, we know $\CCO(S)=(0,0) \in \mathcal{H}$. Hence $\lim_{k \rightarrow \infty} \CCO( \{ x^{(k)}_{1}, x^{(k)}_{2},x^{(k)}_{3}\}) = \varnothing \neq (0,0) = \CCO(S)$.
\end{proof}

The following question now naturally arises:

\begin{question} \label{quest:CCS:Contin}
Suppose that $\CCO( \{x_{1},x_{2},x_{3}\}) \in \mathcal{H}$, 
and let $\big(
(x^{(k)}_{1}, x^{(k)}_{2},x^{(k)}_{3}) \big)_{k \geq 1} \subseteq
\mathcal{H}^{3}$ be such that $\lim_{k \rightarrow \infty} (x^{(k)}_{1},
x^{(k)}_{2},x^{(k)}_{3})=(x_{1},x_{2},x_{3})$. Is it true that
the implication 
\begin{equation*}
 (\forall k \geq 1) \quad \CCO(\{x^{(k)}_{1}, x^{(k)}_{2},x^{(k)}_{3}\}) \in \mathcal{H}  \Longrightarrow \lim_{k \rightarrow \infty} \CCO(\{x^{(k)}_{1}, x^{(k)}_{2},x^{(k)}_{3}\}) = \CCO( \{x_{1},x_{2},x_{3}\}) 
\end{equation*}
holds?
\end{question}
When $x_{1},x_{2}, x_{3}$ are affinely independent, then 
\cref{thm:CCO:LimitCont}\cref{prop:CCO:LimitCont:LinearIndep}
gives us an affirmative answer. 
However, the answer is negative if 
$x_{1},x_{2}, x_{3}$ are not assumed to be affinely independent.

\begin{example} \label{exam:CounterExampleContinuity }
Suppose that 
$\mathcal{H}=\mathbb{R}^{2}$ and $S=\{x_{1}, x_{2},x_{3}\}$ with $x_{1}=(-2,0)$, $x_{2}=x_{3}=(2,0)$. 
Then there exists a sequence 
$\big( (x^{(k)}_{1}, x^{(k)}_{2},x^{(k)}_{3}) \big)_{k \geq 1} \subseteq \mathcal{H}^{3}$ such that
\begin{enumerate}
\item \label{exam:CounExamCont:condi} $\lim_{k \rightarrow \infty} (x^{(k)}_{1}, x^{(k)}_{2},x^{(k)}_{3})=(x_{1},x_{2},x_{3})$,
\item \label{exam:CounExamCont:i} $(\forall k \geq 1) \quad \CCO(\{x^{(k)}_{1}, x^{(k)}_{2},x^{(k)}_{3}\}) \in \mathbb{R}^{2}$, and
\item \label{exam:CounExamCont:ii} $\lim_{k \rightarrow \infty} \CCO( \{ x^{(k)}_{1}, x^{(k)}_{2},x^{(k)}_{3}\}) \neq \CCO(S)$.
\end{enumerate}
\end{example}
\begin{proof}
By \cref{defn:Circumcenter}, we know that $\CCO(S)=(0,0)\in \mathcal{H}$. 
Set 
\begin{align*}
(\forall k \geq 1) \quad S^{(k)}=\{x^{(k)}_{1},
x^{(k)}_{2},x^{(k)}_{3}\}=\Big\{(-2,0), (2,0),
\big(2-\tfrac{1}{k},\tfrac{1}{4k}\big) \Big\}.
\end{align*} 
\cref{exam:CounExamCont:condi}: In this case,
\begin{align*}
\lim_{k \rightarrow \infty} (x^{(k)}_{1},
x^{(k)}_{2},x^{(k)}_{3}) =&\lim_{k \rightarrow \infty} \Big(
(-2,0), (2,0), \big(2-\tfrac{1}{k},\tfrac{1}{4k}\big) \Big)\\
=&\big( (-2,0),(2,0),(2,0) \big) \\
=&(x_{1},x_{2},x_{3}).
\end{align*}
\cref{exam:CounExamCont:i}: 
It is clear that for every  $k \geq 1$, 
the vectors $(-2,0), (2,0), (2-\frac{1}{k},\frac{1}{4k}) $ are
not colinear, that is, $(-2,0), (2,0),
(2-\frac{1}{k},\frac{1}{4k})$ are affinely independent. By
\cref{thm:unique:LinIndpPformula}, we see that
\begin{align*}
(\forall k \geq 1) \quad \CCO(\{x^{(k)}_{1}, x^{(k)}_{2},x^{(k)}_{3}\}) \in \mathbb{R}^{2}.
\end{align*}
\cref{exam:CounExamCont:ii}: Let $k \geq 1$. By definition of
$\CCO(S^{(k)})$ and \cref{exam:CounExamCont:i}, we deduce that 
$\CCO(S^{(k)}) =(p^{(k)}_{1}, p^{(k)}_{2}) \in \mathbb{R}^{2}$ and
that 
\begin{align*}
\norm{\CCO(S^{(k)}) -x^{(k)}_{1} }= \norm{\CCO(S^{(k)}) -x^{(k)}_{2} }=\norm{\CCO(S^{(k)}) -x^{(k)}_{3} }.
\end{align*} 
Because $ \CCO(S^{(k)})$ must be in the intersection of the perpendicular bisector of $ x^{(k)}_{1}=(-2,0), x^{(k)}_{2}=(2,0)$ and the perpendicular bisector of $ x^{(k)}_{2}=(2,0), x^{(k)}_{3}=(2-\frac{1}{k},\frac{1}{4k})$, 
we obtain 
\begin{align*}
p^{(k)}_{1}=0 \quad \text{and} 
\quad p^{(k)}_{2}=4\big( p^{(k)}_{1} -
\frac{2+2-\frac{1}{k}}{2}\big)+\tfrac{1}{8k};
\end{align*}
thus, 
\begin{align} \label{eq:exam:CounterExampleContinuity:SkFormula}
 \CCO(S^{(k)})=(p^{(k)}_{1}, p^{(k)}_{2}) =
 \big(0, -8+\tfrac{2}{k}+\tfrac{1}{8k}\big).
 \end{align}
(Alternatively, we can use the formula in \cref{thm:unique:LinIndpPformula} to get \cref{eq:exam:CounterExampleContinuity:SkFormula}). 
Therefore, 
\begin{align*}
\lim_{k \rightarrow \infty} \CCO(S^{(k)})= \lim_{k \rightarrow
\infty} \big(0, -8+\tfrac{2}{k} +\tfrac{1}{8k}\big)=(0,-8) \neq
(0,0)=\CCO(S), 
\end{align*}
and the proof is complete. 

As the picture below shows, $(\forall k \geq 1)$  $x^{(k)}_{3}=(2-\frac{1}{k},\frac{1}{4k})$ converges to $x_{3}=(2,0)$ along the purple line $L=\{(x,y) \in \mathbb{R}^{2}~|~y=-\frac{1}{4}(x-2) \}$. In fact, $\CCO(S^{(k)})$ is just the intersection point between the two lines $M_{1}$ and $M_{2}$, where $M_{1}$ is the perpendicular bisector between the points $x_{1}$ and $x_{2}$, and $M_{2}$ is the perpendicular bisector between the points $x^{(k)}_{3}$ and $x_{2}$.
\begin{figure}[H] \label{CountExampContinu.png}
\begin{center}\includegraphics[scale=0.6]{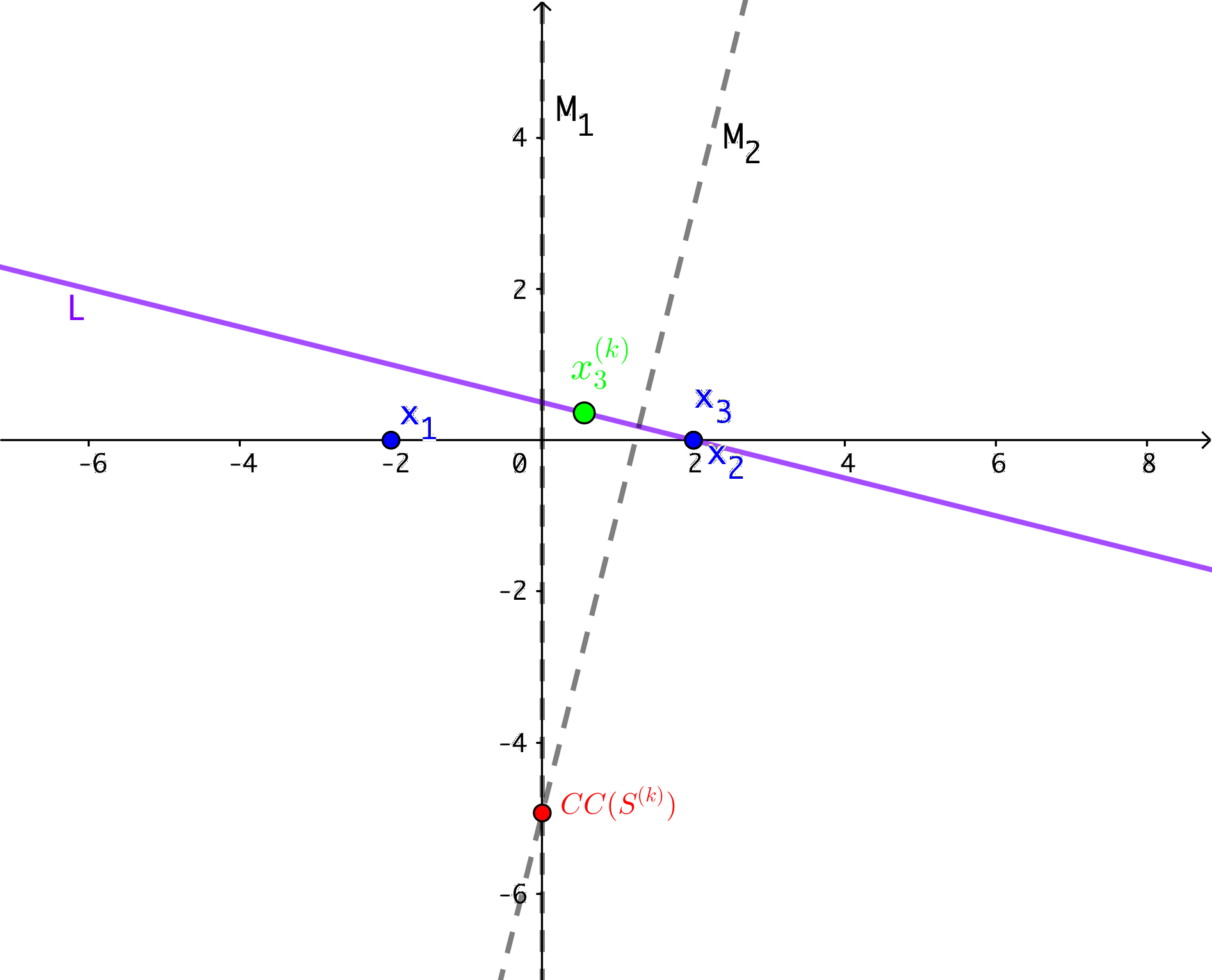}
\end{center}
\caption{Continuity of circumcenter operator may fail even when
$(\forall k \geq 1)$ $\CCO(S^{(k)}) \in \mathcal{H}$.}
\end{figure}
\end{proof}

\section{The circumcenter of three points} 
\label{sec:CircThreePoints}

In this section, we study the circumcenter of a set containing 
three points. 
We will give a characterization of the existence of circumcenter
of three pairwise distinct points. In addition, we shall provide 
asymmetric and symmetric formulae.

\begin{theorem}\label{thm:clform:three}
Suppose that 
$S=\{x,y,z\} \in \mathcal{P}(\mathcal{H})$ and that  $l=3$ 
is the cardinality of $S$. 
Then $x, y, z$ are affinely independent if and only if 
$\CCO(S) \in \mathcal{H}$. 
\end{theorem}

\begin{proof}
If $S$ is affinely independent, then $\CCO(S) \in \mathcal{H}$ by 
\cref{thm:unique:LinIndpPformula}. 

To prove the converse implication, suppose that 
$\CCO(S) \in \mathcal{H}$, i.e.,
\begin{enumerate}
\item $\CCO(S) \in \aff \{x,y,z\}$ , and
\item $\norm{\CCO(S)-x}=\norm{\CCO(S)-y}=\norm{\CCO(S)-z}$.
\end{enumerate}
We argue by contradiction and thus assume that 
the elements of $S$ are affinely dependent:
\begin{align*}
\dim( \spn\{S-x\}) = \dim ( \spn \{y-x, z-x\}) \leq 1.
\end{align*}
Note that $y-x \neq 0$ and $z-x \neq 0$. Set
\begin{align*}
U = x + \spn \{y-x, z-x\} = x +\spn \{y-x\} = x+\spn \{z -x\}.
\end{align*}
Combining with \cref{lem:AffineHull}, we get
\begin{align} \label{eq:thm:three:URepre}
U =\aff \{x, y,z\} = \aff\{x,y\} =\aff\{x,z\}.
\end{align}
By definition of $\CCO(S)$, we have 

\begin{align} \label{eq:thm:clform:three:p1}
\CCO(S) \in \aff \{x,y\} \stackrel{\cref{eq:thm:three:URepre}}{=} U \quad \mbox{and} \quad \norm{\CCO(S)-x}=\norm{\CCO(S)-y},
\end{align}
and 
\begin{align} \label{eq:thm:clform:three:p2}
\CCO(S) \in \aff \{x,z\} \stackrel{\cref{eq:thm:three:URepre}}{=}  U \quad \mbox{and} \quad \norm{\CCO(S)-x}=\norm{\CCO(S)-z}.
\end{align}
Now using \cref{prop:NormEqInnNorm:Norm} $\Leftrightarrow$ \ref{prop:NormEqInnNorm:ProjeEqu} in \cref{prop:NormEqInnNorm} and using \cref{eq:thm:clform:three:p1}, we get
\begin{align*}
\CCO(S)=P_{U}\Big(\CCO(S) \Big)=\frac{x+y}{2}.
\end{align*}
Similarly,  using \cref{prop:NormEqInnNorm:Norm} $\Leftrightarrow$ \cref{prop:NormEqInnNorm:ProjeEqu} in \cref{prop:NormEqInnNorm} and using  \cref{eq:thm:clform:three:p2}, we can also get 
\begin{align*}
\CCO(S)=P_{U}\Big(\CCO(S) \Big)=\frac{x+z}{2}.
\end{align*}
Therefore, 
\begin{align*}
\frac{x+y}{2} =\CCO(S)=\frac{x+z}{2} 
\Longrightarrow y=z,
\end{align*}
which contradicts the assumption that  $l=3$.  
The proof is complete. 
\end{proof}

In contrast, when 
the cardinality of $S$ is $4$, then 
\begin{empheq}[box=\mybluebox]{equation*}
\CCO(S) \in \mathcal{H} \not \Rightarrow ~\text{elements of $S$
are affinely independent} 
\end{empheq}
as the following example demonstrates. 
Thus the characterization of the existence of circumcenter in
\cref{thm:clform:three} is generally not true when we consider
$l\geq 3$ pairwise distinct points.

\begin{example} \label{exam:AffinDepExist}
Suppose that 
$\mathcal{H}=\mathbb{R}^{2}$, that $m=4$, and 
$S = \{x_{1},x_{2}, x_{3},x_{4}\}$,
where $x_{1}=(0,0)$, $x_{2}=(4,0)$, $x_{3}=(0,4)$, and $x_{4}=(4,4)$ 
(see \cref{CircFourAffineDep.png}). 
Then $x_{1},x_{2}, x_{3},x_{4}$ are pairwise distinct and affinely dependent, yet $\CCO(S) = (2,2)$.
\end{example}

\begin{figure}[H]
\begin{center}\includegraphics[scale=0.8]{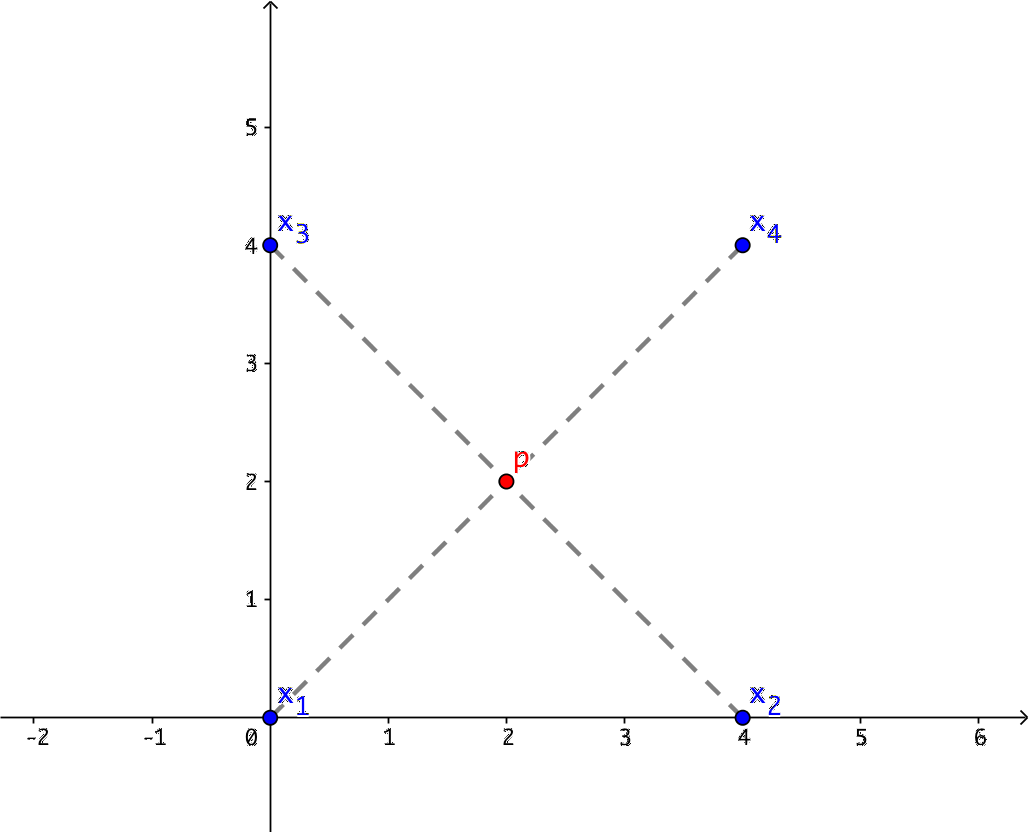}
\end{center}
\caption{Circumcenter of the four affinely dependent points from
\cref{exam:AffinDepExist}.} \label{CircFourAffineDep.png}
\end{figure}

In \cref{prop:unique:PExisUnique} and \cref{thm:unique:LinIndpPformula} above, 
where we presented formulae for $\CCO(S)$, we  gave special
importance to the first point $x_{1}$ in $S$. 
We now provide some longer yet symmetric formulae for $\CCO(S)$.

\begin{remark}\label{rem:SymFormCircu}
Suppose that $S=\{x,y,z\}$ and that $l=3$ is the cardinality of $S$. 
Assume furthermore that $\CCO(S) \in \mathcal{H}$, i.e., there is 
an unique point $\CCO(S)$ satisfying 
\begin{enumerate}
\item $\CCO(S) \in \aff \{x,y,z\}$  and
\item $\norm{\CCO(S)-x}=\norm{\CCO(S)-y}=\norm{\CCO(S)-z}$.
\end{enumerate}
By \cref{thm:clform:three}, the vectors $x, y, z$ must be
affinely independent. 
From \cref{thm:unique:LinIndpPformula} we obtain 
\begin{align*}
\CCO(S) = &~ x+\frac{1}{2} (y-x,z-x)\begin{pmatrix} 
\norm{y-x}^{2} & \innp{y-x, z-x}  \\ 
\innp{z-x,y-x} & \norm{z-x}^{2}  \\
\end{pmatrix} ^{-1}
\begin{pmatrix}
 \norm{y-x}^{2} \\
 \norm{z-x}^{2} \\
\end{pmatrix} \\
= &~ x+\frac{(\norm{y-x}^{2} \norm{z-x}^{2} -\norm{z-x}^{2}\innp{y-x, z-x})(y-x)}{2(\norm{y-x}^{2}\norm{z-x}^{2}- \innp{y-x, z-x}^{2})} \\
&~ +\frac{ (\norm{y-x}^{2} \norm{z-x}^{2} -\norm{y-x}^{2}\innp{y-x, z-x})(z-x) }{2(\norm{y-x}^{2}\norm{z-x}^{2}- \innp{y-x, z-x}^{2})} \\
= &~ \frac{1}{K_{1}}\Big(  \norm{y-z}^{2} \innp{x-z,x-y}x+ \norm{x-z}^{2} \innp{y-z,y-x}y+ \norm{x-y}^{2} \innp{z-x,z-y}z \Big),
\end{align*}
where $K_{1}=2(\norm{y-x}^{2}\norm{z-x}^{2}- \innp{y-x, z-x}^{2})$.\\
Similarly, 
\begin{align*}
\CCO(S)=\frac{1}{K_{2}}\Big(  \norm{y-z}^{2} \innp{x-z,x-y}x+\norm{x-z}^{2} \innp{y-z,y-x}y+ \norm{x-y}^{2} \innp{z-x,z-y}z \Big),
\end{align*}
where $K_{2}=2(\norm{x-y}^{2}\norm{z-y}^{2}- \innp{x-y, z-y}^{2})$ and
\begin{align*}
\CCO(S)=\frac{1}{K_{3}}\Big(  \norm{y-z}^{2} \innp{x-z,x-y}x+ \norm{x-z}^{2} \innp{y-z,y-x}y+ \norm{x-y}^{2} \innp{z-x,z-y}z \Big),
\end{align*}
where $K_{3}=2(\norm{x-z}^{2}\norm{y-z}^{2}- \innp{x-z, y-z}^{2})$.

In view \cref{prop:unique:PExisUnique} (the uniqueness of the
circumcenter), we now average the three formulae from above to
obtain the following symmetric formula for $p$: 
\begin{align*}
\CCO(S)= \frac{1}{K}\Big(  \norm{y-z}^{2} \innp{x-z,x-y}x+ \norm{x-z}^{2} \innp{y-z,y-x}y+\norm{x-y}^{2} \innp{z-x,z-y}z \Big),
\end{align*}
where 
$K=\frac{1}{6} \Big( \frac{1}{\norm{y-x}^{2}\norm{z-x}^{2}- \innp{y-x, z-x}^{2}}+\frac{1}{\norm{x-y}^{2}\norm{z-y}^{2}- \innp{x-y, z-y}^{2}}+\frac{1}{\norm{x-z}^{2}\norm{y-z}^{2}- \innp{x-z, y-z}^{2}} \Big)$.
In fact, \cref{prop:GramMatrSymm} yields $K_{1}=K_{2}=K_{3}$.
\end{remark}

We now summarize the above discussion so far in the following two
pleasing main results. 

\begin{theorem}[nonsymmetric formula for the circumcenter] 
\label{thm:SymForm}
Suppose that $S=\{x,y,z\}$ and denote the cardinality of $S$ by
$l$. 
Then exactly one of the following cases occurs: 
\begin{enumerate}
\item \label{thm:SymForm:1} $l=1$ and $\CCO(S) =x$.
\item \label{thm:SymForm:2} $l=2$, say $S=\{u,v\}$, where $u,v \in S$ 
and $u \neq v$, and $\CCO(S)=\frac{u+v}{2}$.
\item \label{thm:SymForm:3} 
$l=3$ and exactly one of the following two cases occurs: 
\begin{enumerate}
\item \label{thm:SymForm:3:a} 
$x, y, z$ are affinely independent; equivalently, 
$\norm{y-x}\norm{z-x} > \innp{y-x,z-x}$, and 
\begin{align*}
\CCO(S) = \frac{\norm{y-z}^{2} \innp{x-z,x-y}x+ \norm{x-z}^{2} \innp{y-z,y-x}y+ \norm{x-y}^{2} \innp{z-x,z-y}z }{2(\norm{y-x}^{2}\norm{z-x}^{2}- \innp{y-x, z-x}^{2})}.
\end{align*}
\item \label{thm:SymForm:3:b} 
$x, y, z$ are affinely dependent; equivalently, 
$\norm{y-x}\norm{z-x} = \innp{y-x,z-x}$, and $\CCO(S) =
\varnothing $. 
\end{enumerate}
\end{enumerate}
\end{theorem}

\begin{theorem}[symmetric formula of the circumcenter] \label{thm:Formular:Circum:Allcases}
Suppose that $S=\{x,y,z\}$ and 
denote the cardinality of $S$ by $l$. 
Then exactly one of the following cases occurs:
\begin{enumerate}
\item \label{thm:Formular:Circum:Allcases:i} 
$l=1$ and $\CCO(S)=x=y=z = \frac{x +y +z}{3}$.
\item \label{thm:Formular:Circum:Allcases:ii} $l=2$ and
$\CCO(S)= \frac{\norm{x-y}z+\norm{x-z}y+\norm{y-z}x}{\norm{x-y} +\norm{x-z}+\norm{y-z}}$.
\item \label{thm:Formular:Circum:Allcases:iii} $l=3$,
consider $K=\frac{1}{6} \big(
\frac{1}{\norm{y-x}^{2}\norm{z-x}^{2}- \innp{y-x,
z-x}^{2}}+\frac{1}{\norm{x-y}^{2}\norm{z-y}^{2}- \innp{x-y,
z-y}^{2}}+\frac{1}{\norm{x-z}^{2}\norm{y-z}^{2}- \innp{x-z,
y-z}^{2}} \big)$,
and exactly one of the following two cases occurs:
\begin{enumerate} 
\item \label{thm:Formular:Circum:Allcases:iii:a} 
$K\in\left]0,+\infty\right[$ and 
\begin{align*}
\CCO(S)
& = 
\frac{\norm{y-z}^{2} \innp{x-z,x-y}x+ \norm{x-z}^{2}
\innp{y-z,y-x}y+\norm{x-y}^{2} \innp{z-x,z-y}z}{K}. 
\end{align*}
\item \label{thm:Formular:Circum:Allcases:iii:b} 
$K$ is not defined (because of a zero denominator) and 
$\CCO(S) = \varnothing $. 
\end{enumerate}
\end{enumerate}
\end{theorem}

\section{Applications of the circumcenter}
\label{sec:applications}

In this section, we discuss applications of the circumcenter in
optimization. 

Let $z \in \mathcal{H}$, and let $U_1,\ldots,U_m$ be closed
subspaces of $\mathcal{H}$. The corresponding 
best approximation problem is to 
\begin{empheq}[box=\mybluebox]{equation} \label{eq:BestAppro}
\mbox{Find}~\bar{u} \in \cap^{m}_{i=1} U_{i} 
~\mbox{such that}~ \norm{z-\bar{u}}=\min_{u \in \cap^{m}_{i=1} U_{i}} 
\norm{z-u}.
\end{empheq}
Clearly, the solution of \cref{eq:BestAppro} is just 
$P_{\cap^{m}_{i=1} U_{i}}z$.

Now assume that $\mathcal{H} = \mathbb{R}^{n}$, and let 
$U$ and $V$ be linear subspaces of $\mathcal{H}$, i.e.,
we focus on $m=2$ subspaces. 
Set 
\begin{align*}
\mathcal{S}\colon \mathbb{R}^n\to\mathcal{P}(\mathbb{R}^n)\colon
x\mapsto \{x, R_{U}x, R_{V}R_{U}x\}.
\end{align*}
Behling, Bello Cruz, and Santos introduced and studied
in \cite{BCS2017} an algorithm 
to accelerate the Douglas--Rachford algorithm 
they termed the 
\emph{Circumcentered-Douglas-Rachford method (C-DRM)}. 
Given a current point $x \in  \mathbb{R}^{n}$, the next iterate
of the C-DRM is the circumcenter of the triangle with vertices 
$x$, $R_{U}x$ and $R_{V}R_{U}x$. 
Hence, given the initial point $x \in  \mathbb{R}^{n}$, 
the C-DRM generates the sequence $(x^{(k)})_{k \in \mathbb{N}}$
via
\begin{align} \label{eq:rem:CDRM}
x^{(0)}=x, \quad \text{and} \quad (\forall k \in \mathbb{N}) \quad x^{(k+1)} =\CCO(\mathcal{S}(x^{(k)})).
\end{align}
Behling et al.'s \cite[Lemma~2]{BCS2017} guarantees that 
for every $x \in \mathbb{R}^{n}$, the circumcenter 
$\CCO(\mathcal{S}(x))$ is the projection of any point $w \in U \cap V$ onto the affine subspace $\aff  \{x, R_{U}x, R_{V}R_{U}x\}$. 
Here, the existence of the circumcenter of $\mathcal{S}(x)$ turns
out to be a necessary condition for the nonemptiness of $U \cap V$. 
In fact, $\CCO(\mathcal{S}(x)) = P_{\aff (\mathcal{S}(x))} (P_{U\cap V}x)$, which means that $\CCO(\mathcal{S}(x))$ is the closest point to the $P_{U\cap V}x$ among the points in the affine subspace $\aff (\mathcal{S}(x))$.
In \cite[Theorem~1]{BCS2017}, the authors proved that 
if $x$ in \cref{eq:rem:CDRM} is replaced by $P_{U}z$, $P_{V}z$ or
$P_{U+V}z$, where $z \in \mathbb{R}^{n}$, then 
the C-DRM sequence defined in \cref{eq:rem:CDRM} converges 
linearly to $P_{U\cap V}z$. 
Moreover, their rate of convergence is at least 
the cosine of the Friedrichs angle between $U$ and $V$, 
$c_{F} \in \left[0,1\right[$, which happens to be the sharp rate for 
the original DRM; see \cite[Theorem~4.1]{BCNPW2014} for details. 

In \cite[Section~3.1]{BCS2017}, the authors elaborate on how to compute
the circumcenter of $\mathcal{S}(x)$ in $\mathbb{R}^{n}$. They used
the fact that the projection of $\CCO(\mathcal{S}(x))$ onto each
vector $R_{U}x -x$ and $R_{V}R_{U}x -x$ has its endpoint at the
midpoint of the line segment from $x$ to $R_{U}x$ and $x$ to
$R_{V}R_{U}x$. They exhibited a $2 \times 2$
linear system of equations to calculate the $\CCO(\mathcal{S}(x))$
and an expression of the $\CCO(\mathcal{S}(x))$ with parameters.
Their expression of the $\CCO(\mathcal{S}(x))$ can be deduced from
our \cref{rem:SymFormCircu}. Actually, for every $x \in \mathbb{R}^{n}$,
using \cref{thm:SymForm}\cref{thm:SymForm:3:a}, we can easily
obtain
a closed formula for $\CCO(\mathcal{S}(x))$ allowing us 
to efficiently calculate the C-DRM sequence.

In \cite[Corollary~3]{BCS2017}, Behling et al.\ proved that 
their linear convergence results are applicable to affine subspaces 
with nonempty intersection using the Friedrichs angle of suitable
linear subspaces parallel to the original affine subspaces. 
Returning to \cref{eq:BestAppro}, we now set
\begin{align*}
\widehat{\mathcal{S}} \colon
\mathbb{R}^n\to\mathcal{P}(\mathbb{R}^n)\colon
x\mapsto \big\{x, R_{U_{1}}x, R_{U_{2}}R_{U_{1}}x, \ldots,
R_{U_{m}}\cdots R_{U_{2}}R_{U_{1}}x\big\}.
\end{align*}
In order to minimize the inherent zig-zag behaviour of sequences
generated by various reflection and projection methods, 
Behling et al.\ generalized the C-DRM in 
\cite{BCS2018} to the so-called
\emph{Circumcentered-Reflection Method (CRM)}. 
Using our notation, it turns out that 
the underlying CRM operator $C\colon\mathbb{R}^n\to\mathbb{R}^n$ 
is nothing but 
the composition $\CCO\,\circ\, \widehat{\mathcal{S}}$.
Hence Behling et al.'s CRM sequence
is just
\begin{align} \label{eq:CRM:Sequ}
x^{(0)}=x, \quad \text{and} \quad (\forall k \in \mathbb{N}) \quad x^{(k+1)} =\CCO(\widehat{\mathcal{S}}(x^{(k)})).
\end{align}
In \cite[Lemma~3.1]{BCS2018}, they show $C$ is well defined.
Moreover, they also obtain 
\begin{align*}
(\forall w \in \cap^{m}_{i=1} U_{i} ) \quad \CCO(\widehat{\mathcal{S}}(x)) = P_{\aff (\widehat{\mathcal{S}}(x))} (w).
\end{align*}
In particular, $\CCO(\widehat{\mathcal{S}}(x)) = P_{\aff
(\widehat{\mathcal{S}}(x))}(P_{\cap^{m}_{i=1} U_{i}}x)$, which means
that the circumcenter of  the set $\widehat{\mathcal{S}}(x)$ is
the point in $\aff ( \widehat{\mathcal{S}}(x) )$ that is closest to
$P_{\cap^{m}_{i=1} U_{i}}x$.
Behling et al.'s central convergence result (see 
\cite[Theorem~3.3]{BCS2018})  states that the CRM sequence 
\cref{eq:CRM:Sequ} converges linearly to $P_{\cap^{m}_{i=1} U_{i}}x$. 

For the actual computation of the circumcenter of the set
$\widehat{\mathcal{S}}(x)$, both \cite{BCS2017} and \cite{BCS2018}
only contain passing references to that the computation 
``requires the resolution of a suitable $m \times m$ linear
system of equations.'' 
Concluding this section, let us point out 
that the explicit formula presented in 
\cref{cor:unique:BasisPformula} may be used; 
after finding a maximally linearly independent subset of
$\widehat{\mathcal{S}}(x)-x$ (using Matlab, say) 
one can directly use the formula in \cref{cor:unique:BasisPformula}
to calculate the circumcenter.

 \section{The circumcenter in $\mathbb{R}^3$ and the crossproduct}\label{sec:CircCrossProd}

We conclude this paper by expressing the circumcenter and
circumradius in $\mathbb{R}^3$ by using the cross product.
We start by  reviewing some properties of the cross product.

\begin{definition}[cross product] 
\cite[page 483]{A1967}  \label{def:CrosPro}
Let $x=(x_{1}, x_{2}, x_{3})$ and $y=(y_{1}, y_{2},y_{3})$ be two vectors in $\mathbb{R}^{3}$. 
The \emph{cross product} $x \times y$ (in that order) 
is 
\begin{align*}
x \times y=(x_{2}y_{3}-x_{3}y_{2}, x_{3}y_{1}-x_{1}y_{3}, x_{1}y_{2}-x_{2}y_{1}).
\end{align*}
\end{definition}

\begin{fact}  
{\rm \cite[Theorem~13.12]{A1967} and \cite[Theorem~17.12]{C1969}} 
\label{lem:CrosProdProper}
Let $x, y, z$ be in $\mathbb{R}^{3}$. Then the following hold:
\begin{enumerate}
\item \label{lem:CrosProdProper:Binliea} The cross product defined in \cref{def:CrosPro} is a bilinear function, that is, for every $\alpha, \beta \in \mathbb{R}$,
\begin{align*}
(\alpha x +\beta y) \times z =\alpha (x \times z) +\beta (y \times z) \quad \text{and} \quad
x \times (\alpha y +\beta z)= \alpha (x \times y) +\beta (x \times z).
\end{align*}
\item \label{lem:CrosProdProper:PerpenSpanSpace} $x \times y  \in (\spn \{x,y\})^{\perp}$, that is
\begin{align*}
(\forall \alpha \in \mathbb{R}) \quad (\forall \beta \in \mathbb{R}) \quad \innp{x \times y , \alpha x +\beta y} =0.
\end{align*}
\item  \label{fact:CroProProp:TripCroPro}
We have 
\begin{align*}
(x \times y) \times z = \innp{x,z}y -\innp{y,z}x \quad \text{and} \quad
x \times (y \times z)= \innp{x,z} y -\innp{x,y}z.
\end{align*}
\item \label{lem:CrosProdProper:Lagrang} 
{\rm \textbf{(Lagrange's identity)}}
$\norm{x \times y}^{2}=\norm{x}^{2}\norm{y}^{2}-\innp{x,y}^{2}$. 
\end{enumerate} 
\end{fact}

\begin{definition} \cite[page 458]{A1967}  \label{def:Angle}
Let $x$ and $y$ be two nonzero vectors in $\mathbb{R}^{n}$, where $n \geq 1$.  
Then the  \emph{angle} $\theta$ between $x$ and $y$ is defined by 
\begin{align*}
\theta =\arccos \frac{\innp{x,y}}{\norm{x}\norm{y}},
\end{align*}
where $\arccos \colon [-1,1] \to [0,\pi]$.
\end{definition}

\begin{remark}
If $x$ and $y$ are two nonzero vectors in $\mathbb{R}^{n}$, where $n \geq 1$, then
\begin{align*}
\innp{x,y}=\norm{x}\norm{y}\cos \theta,
\end{align*} 
where $\theta$ is the angle between $x$ and $y$.
\end{remark}

\begin{fact}
{\rm \cite[page~485]{A1967}} 
\label{fact:CroProProp:crosproGeome}
Let $x$ and $y$ be two nonzero vectors 
in $\mathbb{R}^{3}$, and let $\theta$ be the angle 
between $x$ and $y$. 
Then 
\begin{align*}
\norm{x \times y}=\norm{x}\norm{y}\sin \theta =~\text{the area of the parallelogram determined by $x$ and $y$}.
\end{align*}
\end{fact}

Now we are ready for the expression of
the circumcenter and circumradius by cross product.

\begin{theorem} \label{thm:CrosCircum}
Suppose that $\mathcal{H}=\mathbb{R}^{3}$,
that $x,y,z$ are affinely independent, and that $S=\{x,y,z \} $. 
Set $a=y-x$, and $b=z-x$ and let the angle between $a$ and $b$,
defined in \cref{def:Angle}, be $\theta$. 
Then
\begin{enumerate}
\item \label{thm:CrosCircum:i} $\CCO(S)=x+\frac{(\norm{a}^{2}b-\norm{b}^{2}a )\times (a \times b)}{2 \norm{a \times b}^{2}}$.
\item {\rm\cite[1.54]{C1969}} \label{thm:CrosCircum:ii} $
\CRO(S)=\frac{\norm{a}\norm{b}\norm{a-b}}{2 \norm{a \times b}} =\frac{\norm{a-b}}{2 \sin \theta}$.
\end{enumerate}
\end{theorem}

\begin{proof}
\cref{thm:CrosCircum:i}: Using the formula of circumcenter in \cref{thm:unique:LinIndpPformula}, we have
\begin{align*}
\CCO(S)&=x+\frac{1}{2} \begin{pmatrix}  y-x &z-x \end{pmatrix}\begin{pmatrix} 
\norm{y-x}^{2} & \innp{y-x, z-x}  \\ 
\innp{z-x,y-x} & \norm{z-x}^{2}  \\
\end{pmatrix} ^{-1}
\begin{pmatrix}
 \norm{y-x}^{2} \\
 \norm{z-x}^{2} \\
\end{pmatrix} \\
&=x+\frac{1}{2} \begin{pmatrix} a& b \end{pmatrix}  \begin{pmatrix} 
\norm{a}^{2} & \innp{a,b}  \\ 
\innp{b,a} & \norm{b}^{2}  \\
\end{pmatrix} ^{-1}
\begin{pmatrix}
 \norm{a}^{2} \\
 \norm{b}^{2} \\
\end{pmatrix} \\
&=x+ \frac{1}{2( \norm{a}^{2} \norm{b}^{2}-\innp{a,b}^{2})}  \begin{pmatrix} a& b \end{pmatrix}  
\begin{pmatrix} 
 \norm{b}^{2}  & -\innp{a,b}  \\ 
-\innp{b,a} & \norm{a}^{2} \\
\end{pmatrix} 
\begin{pmatrix}
 \norm{a}^{2} \\
 \norm{b}^{2} \\
\end{pmatrix} \\
&=x+ \frac{1}{2( \norm{a}^{2} \norm{b}^{2}-\innp{a,b}^{2})}  \begin{pmatrix} a& b \end{pmatrix}  
\begin{pmatrix}
 \norm{a}^{2} \norm{b}^{2} - \norm{b}^{2} \innp{a,b}  \\
 \norm{a}^{2} \norm{b}^{2} - \norm{a}^{2}\innp{a,b} \\
\end{pmatrix} \\
&=x+ \frac{ ( \norm{a}^{2} \norm{b}^{2} - \norm{b}^{2} \innp{a,b} )a + ( \norm{a}^{2} \norm{b}^{2} - \norm{a}^{2}\innp{a,b} )b  }{2( \norm{a}^{2} \norm{b}^{2}-\innp{a,b}^{2})}\\
&=x+ \frac{ \innp{\norm{a}^{2}b-\norm{b}^{2}a, b }a - \innp{\norm{a}^{2}b-\norm{b}^{2}a,a}b }{2( \norm{a}^{2} \norm{b}^{2}-\innp{a,b}^{2})}.
\end{align*}
Using the \cref{lem:CrosProdProper} \cref{fact:CroProProp:TripCroPro} and \cref{lem:CrosProdProper:Lagrang}, we get
\begin{align*}
\CCO(S)=x+\frac{(\norm{a}^{2}b-\norm{b}^{2}a )\times (a \times b)}{2 \norm{a \times b}^{2}}.
\end{align*}
\cref{thm:CrosCircum:ii}: By \cref{defn:Circumcenter}, we have
\begin{align} \label{eq:rem:CircuRadius:r}
\CRO(S)=\norm{\CCO(S)-x}=\Norm{ \frac{(\norm{a}^{2}b-\norm{b}^{2}a )\times (a \times b)}{2 \norm{a \times b}^{2}} }.
\end{align}
Using \cref{lem:CrosProdProper}\cref{lem:CrosProdProper:Lagrang}
and \cref{lem:CrosProdProper}\cref{lem:CrosProdProper:PerpenSpanSpace}, 
we obtain 
\begin{align}  \label{eq:radius:Numer}
\Norm{(\norm{a}^{2}b-\norm{b}^{2}a )\times (a \times b)}
& = \Big( \Norm{ \norm{a}^{2}b-\norm{b}^{2}a }^{2} \norm{a \times
b}^{2} - \innp{\norm{a}^{2}b-\norm{b}^{2}a, a \times b}^{2} \Big)^{\frac{1}{2}}\nonumber\\
& = \Norm{\norm{a}^{2}b-\norm{b}^{2}a} \norm{a \times b}. 
\end{align}
In addition, by \cref{note:AffinIndpDetermNonZero}, since $\norm{a} \neq 0$, $\norm{b} \neq 0$, thus
\begin{align} \label{eq:radius:u}
\Norm{ \norm{a}^{2}b-\norm{b}^{2}a} = \norm{a} \norm{b} \NNorm{ \frac{\norm{a}}{\norm{b}} b- \frac{\norm{b}}{\norm{a}} a}.
\end{align}
Now
\begin{align} \label{eq:redius:ThirdPart}
\NNorm{ \frac{\norm{a}}{\norm{b}} b- \frac{\norm{b}}{\norm{a}} a }^{2} 
& = \NNorm{\frac{\norm{a}}{\norm{b}} b }^{2} - 2
\IInnp{\frac{\norm{a}}{\norm{b}} b,  \frac{\norm{b}}{\norm{a}} a }
+ \NNorm{\frac{\norm{b}}{\norm{a}} a}^{2}\nonumber\\
& = \norm{a}^{2} - 2 \innp{b,a} + \norm{b}^{2}\nonumber\\
& = \norm{a-b}^{2}.
\end{align}
Upon combining \cref{eq:radius:Numer}, \cref{eq:radius:u} and \cref{eq:redius:ThirdPart}, we obtain
\begin{align*}
\Norm{(\norm{a}^{2}b-\norm{b}^{2}a )\times (a \times b)} =\norm{a}\norm{b}\norm{a-b}\norm{a \times b}.
\end{align*}
Hence \cref{eq:rem:CircuRadius:r} yields 
\begin{align*}
\CRO(S) & = \frac{1}{2 \norm{a \times b}^{2}}  \Norm{(\norm{a}^{2}b-\norm{b}^{2}a )\times (a \times b)} \\
 & = \frac{1}{2 \norm{a \times b}^{2}} \norm{a}\norm{b}\norm{a-b}\norm{a \times b}\\
& = \frac{\norm{a}\norm{b}\norm{a-b}}{2 \norm{a \times b}} .\\
\end{align*}
By \cref{fact:CroProProp:crosproGeome},  we know 
$\norm{a \times b}= \norm{a}\norm{b} \sin \theta$. 
Thus, we obtain
\begin{align*}
\CRO(S)=\frac{\norm{a}\norm{b}\norm{a-b}}{2 \norm{a \times b}} =\frac{\norm{a-b}}{2 \sin \theta}
\end{align*}
and the proof is complete. 
\end{proof}

\begin{fact} 
{\rm \cite[Theorem I]{M1983}}
\label{fact:CroProDef37}
Suppose that $n \geq 3$, and a cross product is defined 
which assigns to any two vectors $v, w \in \mathbb{R}^{n}$ 
a vector $v \times w \in \mathbb{R}^{n}$ 
such that the following three properties hold:
\begin{enumerate}
\item \label{fact:CroProDef37:bilinear}  $v \times w$ is a bilinear function of $v$ and $w$.
\item \label{fact:CroProDef37:perp} The vector $v \times w$ is perpendicular to both $v$ and $w$. 
 
\item \label{fact:CroProDef37:norm}  $ \norm{v \times w}^{2}=\norm{v}^{2}\norm{w}^{2} -\innp{v,w}^{2}$.
\end{enumerate}
Then $n=3$ or $7$.
\end{fact}

\begin{remark}
In view of \cref{fact:CroProDef37} and our proof of
\cref{thm:CrosCircum}, we cannot generalize the latter result to
a general Hilbert space $\mathcal{H}$ --- unless the dimension of
$\mathcal{H}$ is either 3 or 7. 
\end{remark}

\iffalse
In fact, our  \cref{lem:CrosProdProper}
\cref{lem:CrosProdProper:PerpenSpanSpace} are proved by the two
cross product properties \cref{fact:CroProDef37:bilinear} $\&$
\cref{fact:CroProDef37:perp} in \cref{fact:CroProDef37}, and
\cref{fact:CroProDef37} \cref{fact:CroProDef37:norm} is just our
\cref{lem:CrosProdProper} \cref{lem:CrosProdProper:Lagrang}. Because
in our proof of \cref{thm:CrosCircum} above, we used
\cref{lem:CrosProdProper} \cref{lem:CrosProdProper:PerpenSpanSpace}
$\&$ \cref{lem:CrosProdProper:Lagrang}.  
Hence by \cref{fact:CroProDef37},
if we only use our method of proof for \cref{thm:CrosCircum},  we
cannot generalize our \cref{thm:CrosCircum} to general $\mathcal{H}$
except that the dimension of $\mathcal{H}$ is 3 or 7.
\fi

\section*{Acknowledgments}
HHB and XW were partially supported by NSERC Discovery Grants.

\bibliographystyle{abbrv}

\end{document}